\documentclass{amsart}
\usepackage[utf8]{inputenc}
\usepackage[english]{babel}
\usepackage{csquotes}
\usepackage[backend=biber]{biblatex}
\usepackage{amsmath}
\usepackage{amssymb}
\usepackage{amsthm}

\usepackage{enumerate}
\usepackage{xcolor}
\usepackage{graphicx}
\usepackage{caption}
\usepackage{subcaption}
\usepackage{textcomp}
\usepackage{fancyhdr}

\allowdisplaybreaks

\newtheorem{theorem}{Theorem}[section]
\newtheorem*{theorem*}{Theorem}
\newtheorem{corollary}[theorem]{Corollary}
\newtheorem{lemma}[theorem]{Lemma}
\newtheorem{prop}[theorem]{Proposition}
\newtheorem*{prop*}{Proposition}

\newtheorem*{T1}{Theorem~\ref{main-result-sft-relative}}
\newtheorem*{T2}{Theorem~\ref{main-result-sofic-relative}}
\newtheorem*{T3}{Theorem~\ref{sft-close-to-sofic}}

\theoremstyle{definition}
\newtheorem{definition}[theorem]{Definition}

\theoremstyle{remark}
\newtheorem{remark}[theorem]{Remark}

\newcommand{\A}{\mathcal A}

\newcommand{\F}{\mathcal F}
\renewcommand{\P}{\mathcal P}
\renewcommand{\S}{\mathcal S}
\newcommand{\T}{\mathcal T}

\newcommand{\N}{\mathbb N}
\newcommand{\Z}{\mathbb Z}

\DeclareMathOperator{\opint}{int}
\renewcommand{\int}{\opint}
\DeclareMathOperator{\ints}{ints}

\title{
    Subsystem entropies of shifts of finite type and sofic shifts on countable amenable groups
}

\author{Robert Bland}
\address{Robert Bland\\
Department of Mathematics and Statistics\\
University of North Carolina at Charlotte \\
9201 University City Blvd.\\
Charlotte, NC 28223}
\email{rbland5@uncc.edu}

\author{Kevin McGoff}
\address{Kevin McGoff\\
Department of Mathematics and Statistics\\
University of North Carolina at Charlotte \\
9201 University City Blvd.\\
Charlotte, NC 28223}
\email{kmcgoff1@uncc.edu}
\urladdr{https://clas-math.uncc.edu/kevin-mcgoff/}

\author{Ronnie Pavlov}
\address{Ronnie Pavlov\\
Department of Mathematics\\
University of Denver\\
2390 S. York St.\\
Denver, CO 80210}
\email{rpavlov@du.edu}
\urladdr{https://cs.du.edu/~rpavlov/}



\begin{document}

\maketitle

\begin{abstract}
In this work we study the entropies of subsystems of shifts of finite type (SFTs) and sofic shifts on countable amenable groups. We prove that for any countable amenable group $G$, if $X$ is a $G$-SFT with positive topological entropy $h(X) > 0$, then the entropies of the SFT subsystems of $X$ are dense in the interval $[0, h(X)]$. In fact, we prove a ``relative" version of the same result: if $X$ is a $G$-SFT and $Y \subset X$ is a subshift such that $h(Y) < h(X)$, then the entropies of the SFTs $Z$ for which $Y \subset Z \subset X$ are dense in  $[h(Y), h(X)]$. We also establish analogous results for sofic $G$-shifts.
\end{abstract}

\section{Introduction}

Let $G$ be a countable group and let $\A$ be a finite alphabet of symbols. In symbolic dynamics, the central objects of study are the subsystems of the so-called \textit{full shift}, the dynamical system $(\A^G, \sigma)$, where $\sigma$ denotes the action of $G$ on $\A^G$ by translations (Definition \ref{def:subshifts}). \textit{Shifts of finite type} (Definition \ref{def:SFT}) and \textit{sofic shifts} (Definition \ref{def:sofic}) are the most widely studied and well understood examples of symbolic dynamical systems. In each of these cases, the system of interest is completely specified by a finite amount of information. This allows for combinatorial, finitary arguments to be applied to the analysis of the dynamics of such systems.  

\textit{Entropy} is one of the most fundamental invariants of a topological dynamical system. Many fundamental results from classical entropy theory (i.e., in the case where $G = \mathbb Z$) only generalize if $G$ is an \textit{amenable} group (Definition \ref{def:folner}). Amenability allows one to ``approximate" the group by a sequence of finite subsets in a way that is useful for studying dynamics. See Definition \ref{def:entropy} for the definition of the entropy of a symbolic dynamical system on an amenable group.

In general, one would like to understand the structure of the collection of subsystems of a given subshift. In this paper we study the entropies of the SFT subsystems of a given SFT, as well as the entropies of the sofic subsystems of sofic shifts. There are many existing results in the literature in the case where $G = \mathbb Z$. For example, the Krieger Embedding Theorem~\cite{krieger} characterizes the irreducible SFT subsystems of a given irreducible $\mathbb Z$-SFT. Additionally, Lind \cite{lind} has provided an algebraic characterization of the real numbers that are realized as the entropy of a $\mathbb Z$-SFT.

However, the situation is very different in cases where $G \neq \mathbb Z$. Even in the case where $G = \mathbb Z^d$ for $d > 1$, the classes of SFTs and sofic shifts behave quite differently. For example, Boyle, Pavlov, and Schraudner \cite{boyle_pavlov_schraudner} have shown by example that the subsystems of $\Z^d$ sofic shifts can be badly behaved for $d > 1$ (in contrast with the case where $d = 1$). Moreover, Hochman and Meyerovitch \cite{hochman_meyerovitch} have characterized the real numbers that are realized as entropy of a $\mathbb Z^d$-SFT (with $d >1$), but in contrast to the result of Lind mentioned above, the characterization is in algorithmic terms and unavoidably involves concepts from computability and recursion theory. Nonetheless, Desai \cite{desai} has shown that a $\mathbb Z^d$-SFT with positive entropy has a wealth of SFT subsystems (sharpening an earlier result of Quas and Trow \cite{quas_trow}).

\begin{theorem}[\cite{desai}]
\label{desai}
    Let $G = \mathbb Z^d$ for some $d \in \mathbb N$ and let $X$ be a $G$-SFT such that $h(X) > 0$. Then \[
        \bigl\{h(Y) : Y \subset X \text{ and $Y$ is an SFT} \bigr\}
    \] is dense in $[0,h(X)]$.
\end{theorem}

{In recent years, several results of the $G= \mathbb Z$ and $G = \mathbb Z^d$ cases have seen extensions to larger classes of groups, especially amenable groups. To name a few: Barbieri \cite{barbieri} has classified the real numbers that are realized as the entropy of a $G$-SFT for many types of amenable $G$ (extending the result of Hochman and Meyerovitch mentioned above); Frisch and Tamuz \cite{frisch_tamuz} have investigated the (topologically) generic properties of $G$-subshifts for arbitrary amenable $G$; Barbieri and Sablik \cite{barbieri-sablik} have shown how an arbitrary effective $G$-subshift, where $G$ is finitely generated, may be simulated by a $G'$-SFT, where $G'$ is the semidirect product $G' = \mathbb Z^2 \rtimes G$; and Huczek and Kopacz \cite{huczek-kopacz} have (very recently) obtained a partial generalization of Boyle's lower entropy factor Theorem~\cite{boyle} to countable amenable groups with the comparison property. In this vein, we prove the following generalization of Theorem~\ref{desai} to arbitrary countable amenable groups.}

\begin{T1}
    Let $G$ be a countable amenable group, let $X$ be a $G$-SFT, and let $Y \subset X$ be any subsystem such that $h(Y) < h(X)$. Then \[
        \bigl\{h(Z) : Y \subset Z \subset X \text{ and $Z$ is an SFT} \bigr\}
    \] is dense in $[h(Y), h(X)]$.
\end{T1}

Choosing $G = \mathbb Z^d$ and $Y = \varnothing$ in the above theorem recovers the result of Desai (Theorem~\ref{desai} above). Note that a shift space $X \subset \mathcal{A}^G$ has at most countably many SFT subsystems, and therefore the set of entropies of SFT subsystems is at most countable. In this sense, Theorem~\ref{main-result-sft-relative} is ``the most one could hope for." 

\begin{remark} \label{Rmk:FrischTamuz}
{%(NOTE: I decided to just give the sketch here instead of being mysterious. One reason is that they prove a version of our sofic `realization of all entropies' result for SFTs, so somehow presenting details might diminish impression that our results are just weaker versions of theirs. (See paragraph below Theorem~\ref{main-result-sofic-relative}.) I'm happy to discuss though!) 
After a preprint of this work was made public, the authors of \cite{frisch_tamuz} made us aware that a short alternate proof of Theorem~\ref{main-result-sft-relative} can be derived from their main results. Specifically, they prove there that for any countable amenable group $G$ and any real $c \geq 0$, the set of $G$-subshifts with entropy $c$ is dense (in fact residual) within the space of $G$-subshifts with entropy at least $c$ with respect to the Hausdorff topology. This result immediately implies that for any $G$-SFT $X$, there exist $G$-subshifts contained in $X$ that achieve all possible entropies in $[0, h(X)]$; then, some simple approximations with $G$-SFTs (in the sense of our Theorem~\ref{sft-entropy-approx}) can be used to obtain a proof of Theorem~\ref{main-result-sft-relative}.}
\end{remark}

For sofic shifts, we obtain the following result.

\begin{T2}
    Let $G$ be a countable amenable group, let $W$ be a sofic $G$-shift, and let $V \subset W$ be any subsystem such that $h(V) < h(W)$. Then \[
        \bigl\{h(U) : V \subset U \subset W \text { and $U$ is sofic} \bigr\}
    \] is dense in $[h(V), h(W)]$.
\end{T2}

{From this result, we can quickly derive the fact (Corollary~\ref{sofic-all-entropies-realized}) that if $X$ is a sofic $G$-shift, then each real number in $[0,h(X)]$ can be realized as the entropy of some (not necessarily sofic) subsystem of $X$. (Recall that the alternate proof of Theorem~\ref{main-result-sft-relative} described in Remark~\ref{Rmk:FrischTamuz} above relies on a version of this result requiring $X$ to be an SFT.) The tool for proving Theorem~\ref{main-result-sofic-relative} (from Theorem~\ref{main-result-sft-relative})} is provided by the following theorem, which may be of independent interest. We note that this result generalizes another theorem of Desai \cite[Proposition 4.3]{desai}, which addressed the case $G = \mathbb{Z}^d$.

\begin{T3}
    Let $G$ be a countable amenable group and let $W$ be a sofic $G$-shift. For every $\varepsilon > 0$, there exists an SFT $\tilde X$ and a one-block code $\tilde\phi : \tilde X \to W$ such that {the maximal entropy gap of $\tilde\phi$ satisfies} $\mathcal H(\tilde\phi) < \varepsilon$.
\end{T3}

The maximal entropy gap $\mathcal H(\tilde\phi)$ is defined in \textsection 2 (Definition \ref{def:max-entropy-gap}). In particular, this result implies that if $Y$ is sofic and $\varepsilon > 0$, then there is an SFT $X$ that factors onto $Y$ and satisfies $h(X) < h(Y) + \varepsilon$.

{Our proofs of Theorems~\ref{main-result-sft-relative}, \ref{sft-close-to-sofic}, and \ref{main-result-sofic-relative} take the same general approach as the arguments given by Desai for the $G = \mathbb{Z}^d$ case.} However, the extension to the general amenable setting requires substantial new techniques. Indeed, our proofs are made possible by the existence of \textit{exact tilings} (Definition \ref{def:quasitilings}) of the group $G$ that possess nice dynamical properties. Such exact tilings are trivial to find for $\mathbb Z^d$ (by tiling the group using large hypercubes), but for arbitrary amenable groups were only recently constructed by Downarowicz, Huczek, and Zhang \cite{tilings}; {their construction is} the main technical tool employed in this paper.

{As mentioned in Remark~\ref{Rmk:FrischTamuz} above, Theorem~\ref{main-result-sft-relative} can be alternately derived from results in \cite{frisch_tamuz}. We present a self-contained proof here for two reasons. Firstly, we would like to present a direct adaptation of the techniques from \cite{desai}, since it demonstrates the power of the improved tiling results of \cite{tilings}. Secondly, this presentation provides a unified approach to all of our proofs, since our proofs in the sofic setting (where we are not aware of alternative proofs) also rely on tiling-based constructions that are similar to those in our proof of Theorem~\ref{main-result-sft-relative}.
%Secondly, the proof strategies of Theorems~\ref{main-result-sft-relative} and \ref{sft-close-to-sofic} are quite similar, and the proof of the former is a helpful warmup for that of the latter. %(NOTE: second sentence maybe isn't so convincing. What do you think? Also, I had to back off of the constructive/non-constructive argument; looking again at their proofs, I think they very clearly construct subshifts in a very similar way to our proofs. If one of you sees a way to thread this needle, feel free to go for it.)}
}

% 1. Goal was to provide generalization of Desai's proof
%2. Tool for that was DHZ paper (so we provide application of their work)
%3. Tool for that was DHZ paper (so we provide application of their work)
%4. constructive vs. non-constructive
%1.5. This proof gives a natural generalization of Desai's by applying DHZ

The paper is organized as follows. In \textsection 2 we discuss basic notions and elementary theorems of symbolic dynamics, set in terms appropriate for countable amenable groups. In \textsection 3 we define and explore the concept of tilings and exact tilings of amenable groups, appealing to Downarowicz, Huczek, and Zhang for the existence of certain desirable tilings. In \textsection 4 we prove our main results for $G$-SFTs, and in \textsection 5 we prove our main results for sofic $G$-shifts. Finally, in \textsection 6 we provide a example of a $\mathbb{Z}^2$ sofic shift whose only SFT subsystem is a fixed point.

\section{Basics of symbolic dynamics}

\subsection{Amenable groups}

We begin with a brief overview of amenable groups. 

\begin{definition}[Group theory notations]
    Let $G$ be a group and let $K$, $F \subset G$ be subsets. We employ the following notations.
    \begin{enumerate}[i.]
        \item The \textit{group identity} is denoted by the symbol $e \in G$,
        \item $KF = \{ kf : k \in K \text{ and } f\in F\}$,
        \item $K^{-1} = \{k^{-1} : k \in K\}$,
        \item $Kg = \{kg : k \in K\}$ for each $g \in G$,
        \item $K \sqcup F$ expresses that $K$ and $F$ are disjoint, and is their \textit{(disjoint) union},
        \item $K \triangle F = (K \setminus F) \sqcup (F \setminus K)$ is the \textit{symmetric difference} of $K$ and $F$, and
        \item $|K|$ is the \textit{cardinality} of the (finite) set $K$.
    \end{enumerate}
\end{definition}

\begin{definition}[F\o lner condition for amenability]
\label{def:folner}
    Let $G$ be a countable group. A \textit{F\o lner sequence} is a sequence $(F_n)_{n}$ of finite subsets $F_n \subset G$ which \textit{exhausts} $G$ (in the sense that for each $g \in G$, we have $g \in F_n$ for all sufficiently large $n$) and for which it holds that \[
        \lim_{n\to\infty} \frac{|KF_n \triangle F_n|}{|F_n|} = 0
    \] for every finite subset $K \subset G$. If such a sequence exists, then $G$ is said to be an \textit{amenable} group.
\end{definition}

Throughout this paper, $G$ denotes a fixed countably infinite amenable group and $(F_n)_{n}$ is a fixed F\o lner sequence for $G$.

\begin{definition}[Invariance]
    Let $K$, $F \subset G$ be finite subsets, and let $\varepsilon > 0$. We say $F$ is $(K,\varepsilon)$-\textit{invariant} if\[
        \frac{|KF\triangle F|}{|F|} < \varepsilon.
    \]
\end{definition}

If $e \in K$ and $F$ is $(K,\varepsilon)$-invariant, then $F$ is also $(K',\varepsilon')$-invariant for any $\varepsilon' > \varepsilon$ and any $K' \subset K$ such that $e\in K'$. If $F$ is $(K,\varepsilon)$-invariant, then so is the translate $Fg$ for each fixed $g \in G$.
Invariance is the primary way by which we say a large finite subset $F \subset G$ is a ``good finite approximation" of $G$, according to the finitary quantifiers $K$ and $\varepsilon$. The amenability of $G$ provides a wealth of nearly invariant sets, which enables such approximation for the purpose of studying the dynamics of $G$-actions. 

Next we develop concepts related to the geometry of finite subsets of $G$.

\begin{definition}[Boundary and interior]
\label{def:bdry+int}
    Let $K$, $F \subset G$ be finite subsets. The $K$-\textit{boundary} of $F$ is the set\[
        \partial_K F = \{f \in F : Kf \not\subset F\},
    \] and the $K$-\textit{interior} of $F$ is the set\[
        \int_K F = \{f \in F : Kf \subset F\}.
    \] Observe that $F = (\partial_K F) \sqcup (\int_K F)$.
\end{definition}

If $F$ is sufficiently invariant with respect to $K$, then the $K$-boundary of $F$ is a small subset of $F$ (proportionally), by the following lemma.

\begin{lemma}
\label{small-bdry}
    Suppose $K$, $F \subset G$ are nonempty finite subsets and $e \in K$. Then\[
        \frac{1}{|K|} |KF\triangle F| \leq |\partial_K F| \leq |K||KF \triangle F|.
    \] In particular, if $F$ is $(K,\varepsilon)$-invariant then $|\partial_K F| < \varepsilon|K||F|$.
\end{lemma}

\begin{proof}
    If $e\in K$, then $KF \triangle F = KF \setminus F$. If $g \in KF \setminus F$, then $g = kf$ for some $k \in K$ and $f \in \partial_K F$, by Definition \ref{def:bdry+int}. Therefore $KF \setminus F \subset K \partial_K F$, in which case $|KF\setminus F| \leq |K||\partial_K F|$.
    
    For the second inequality, note that $f \in \partial_K F$ implies $\exists k \in K$ such that $kf \not\in F$, therefore $g = kf \in KF \setminus F$ is a point such that $f\in K^{-1}g \subset K^{-1}(KF\setminus F)$.  Consequently $\partial_K F \subset K^{-1}(KF\setminus F)$, in which case $|\partial_K F| \leq |K||KF \setminus F|$.
    
    Finally if $F$ is $(K,\varepsilon)$-invariant, then $|\partial_KF| \leq |K||KF \setminus F| < \varepsilon|K||F|$.
\end{proof}

Given finite subsets $K$, $F \subset G$, in this paper we focus on the $KK^{-1}$-boundary and $KK^{-1}$-interior of $F$ (rather than the $K$-boundary and $K$-interior), and we make use of the following lemma. %so that the following lemma may be applied.

\begin{lemma}
\label{int-complement}
    Let $K$, $F \subset G$. For any translate $Kg$ of $K$ (for any $g \in G$), either $Kg \subset F$ or $Kg \subset \big(\int_{KK^{-1}} F\big)^c$ (or both are true).
\end{lemma}

\begin{proof}
    Suppose $Kg \not\subset \big(\int_{KK^{-1}} F\big)^c$. Then $\exists f \in \int_{KK^{-1}} F$ such that $f \in Kg$, which implies $g \in K^{-1} f$ and hence $Kg \subset KK^{-1}f \subset F$.
\end{proof}

\subsection{Shift spaces}

Here we present necessary definitions from symbolic dynamics. See Lind and Marcus \cite{lind_marcus} for an introductory treatment of these concepts.

\begin{definition}[Shifts and subshifts]
\label{def:subshifts}
    Let $\A$ be a finite set of symbols equipped with the discrete topology. A function $x : G \to \A$ is called an $\A$-\textit{labelling} of $G$. By convention, we write $x_g$ for the symbol $x(g) \in \A$ which is placed by $x$ at $g \in G$. The set of all $\A$-labellings of $G$ is denoted $\A^G$, which we equip with the product topology. For each $g \in G$, let $\sigma^g : \A^G \to \A^G$ denote the map given by\[
        (\sigma^g x)_h = x_{hg} \quad \forall h \in G
    \] for each $x\in \A^G$. The collection $\sigma = (\sigma^g)_{g\in G}$ is an action of $G$ on $\A^G$ by homeomorphisms. The pair $(\A^G, \sigma)$ is a dynamical system called the \textit{full shift} over the {alphabet} $\A$. A subset $X \subset \A^G$ is called \textit{shift-invariant} if $\sigma^g x \in X$ for each $x \in X$ and $g \in G$. A closed, shift-invariant subset $X \subset \A^G$ is called a \textit{subshift} or a \textit{shift space}. For a given $x \in \A^G$, the \textit{orbit} of $x$ is the subset $\mathcal O(x) = \{\sigma^g x : g \in G\} \subset \A^G$. The subshift \textit{generated by} $x$ is the topological closure of $\mathcal O(x)$ as a subset of $\A^G$, and is denoted $\overline{\mathcal O}(x) \subset \A^G$.
\end{definition}

\begin{definition}[Codes and factors]
\label{def:block-codes}
    Let $\A_X$, $\A_W$ be finite alphabets and let $X \subset \A_X^G$ and $W \subset \A_W^G$ be subshifts. A map $\phi : X \to W$ is \textit{shift-commuting} if $\phi \circ \sigma^g = \sigma^g \circ \phi$ for each $g \in G$; the map $\phi$ is said to be a \textit{sliding block code} if it is continuous and shift-commuting; and $\phi$ is said to be a \textit{factor map} if it is a surjective sliding block code. If a factor map exists from $X$ to $W$, then $W$ is said to be a \textit{factor} of $X$ and $X$ is said to \textit{factor onto} $W$. If a sliding block code $\phi : X \to W$ is invertible and bi-continuous, then $\phi$ is said to be a \textit{topological conjugacy}, in which case $X$ and $W$ are said to be \textit{topologically conjugate}.
\end{definition}

\begin{definition}[Products of shifts]
\label{def:product-of-shifts}
    If $\A$ and $\Sigma$ are finite alphabets, then $\A \times \Sigma$ is also a finite alphabet (of ordered pairs). If $X \subset \A^G$ and $T \subset \Sigma^G$ are subshifts, then we view the \textit{dynamical direct product} $X \times T$ as a subshift of $(\A \times \Sigma)^G$, defined by $(x,t) \in X \times T$ if and only if $x \in X$ and $t \in T$. The shift space $X \times T$ factors onto both $X$ and $T$ via the projection maps $\pi_X$ and $\pi_T$, given by $\pi_X(x,t) = x$ and $\pi_T(x,t) = t$ for each $(x,t) \in X\times T$.
\end{definition}

\begin{remark}
    Definition \ref{def:product-of-shifts} above introduces an abuse of notation, as technically we have $(x,t) \in \A^G \times \Sigma^G \neq (\A \times \Sigma)^G$. However, if equipped with the $G$-action $\varsigma$ given by $\varsigma^g(x,t) = (\sigma^g x,\, \sigma^g t)$, then $\A^G \times \Sigma^G$ becomes a dynamical system that is topologically conjugate to $(\A \times \Sigma)^G$.
\end{remark}

\subsection{Patterns}

In this section we describe \textit{patterns} and their related combinatorics.

\begin{definition}[Patterns]
    Let $\A$ be a finite alphabet and let $F \subset G$ be a finite set. A function $p : F \to \A$ is called a \textit{pattern}, said to be of \textit{shape} $F$. The set of all patterns of shape $F$ is denoted $\A^F$. The set of all patterns of any finite shape is denoted $\A^* = \bigcup_F \A^F$, where the union is taken over all finite subsets $F \subset G$.
\end{definition}

\begin{remark}
    Given a point $x \in \A^G$ and a finite subset $F \subset G$, we take $x(F)$ to mean the \textit{restriction} of $x$ to $F$, which is itself a pattern of shape $F$. Usually this is denoted $x|_F \in \A^F$, but we raise $F$ from the subscript for readability.
\end{remark}

\begin{definition}[One-block code]
    Let $\A_X$ and $\A_W$ be finite alphabets and let  $X \subset \A_X^G$ and $W \subset \A_W^G$ be subshifts. A factor map $\phi : X \to W$ is said to be a \textit{one-block code} if there exists a function $\Phi : \A_X \to \A_W$ with the property that \[
        \phi(x)_g = \Phi(x_g), \quad \forall g\in G
    \] for each $x \in X$.
\end{definition}

\begin{definition}[Occurrence]
\label{def:occurrence}
    Let $\A$ be a finite alphabet and let $F \subset G$ be a finite set.
    A pattern $p \in \A^F$ is said to \textit{occur} in a point $x \in \A^G$ if there exists an element $g \in G$ such that $(\sigma^g x)(F) = p$. If $X \subset \A^G$ is a subshift, then the collection of all patterns of shape $F$ occurring in any point of $X$ is denoted by \[
        \P(F,X) = \{(\sigma^g x)(F) \in \A^F : x\in X \text{ and } g\in G\}.
    \]
\end{definition}    
    
    If $X\subset \A^G$ is a subshift and $F \subset G$ is a finite subset, then $|\P(F,X)| \leq |\A|^{|F|}$. If $F' \subset G$ is another finite subset, then $|\P(F\cup F', X)| \leq |\P(F,X)| \cdot |\P(F', X)|$. If $F' \subset F$ and $X' \subset X$, then $|\P(F',X')| \leq |\P(F,X)|$.
    
\begin{definition}[Forbidden patterns]
    Let $\A$ be a finite alphabet, let $F \subset G$ be a finite set and let $X \subset \A^G$ be a subshift. A pattern $p \in \A^F$ is said to be \textit{allowed} in $X$ if $p \in \P(F,X)$ (if $p$ occurs in at least one point of $X$).
    
    Given a (finite or infinite) collection of patterns $\F \subset \A^*$, a new subshift $X' \subset X$ may be constructed by expressly \textit{forbidding} the patterns in $\F$ from occurring in points of $X$. We denote this by \[
        X' = \mathcal R(X,\F) = \{x \in X : \forall p \in \F, \ \text{$p$ does not occur in $x$}\}.
    \] For a single pattern $p$, we abbreviate $\mathcal R(X, \{p\})$ as $X \setminus p$. The shift $X$ is said to be \textit{specified} by the collection $\F$ if $X = \mathcal R(\A^G, \F)$.
\end{definition}

\subsection{Shifts of finite type}

In this section, we define shifts of finite type and sofic shifts over $G$. We also discuss many related elementary facts.

\begin{definition}[SFTs]
\label{def:SFT}
    A subshift $X \subset \A^G$ is a \textit{shift of finite type (SFT)} if there is a finite collection $\F \subset \A^*$ such that $X = \mathcal R(\A^G, \F)$. For an SFT, it is always possible to take $\F$ in the form $\F = \A^K \setminus \P(K,X)$ for some large finite subset $K \subset G$. In this case, we  say $X$ is \textit{specified} by (patterns of shape) $K$.
\end{definition}

If $X \subset \A^G$ is an SFT specified by a finite subset $K \subset G$, then it holds that \[
    x \in X \iff \forall g\in G\ \big( (\sigma^g x)(K) \in \P(K,X) \big)
\] for each $x \in \A^G$. If $K$ specifies $X$, then so does $K'$ for any (finite) subset $K' \supset K$. If $X$ and $T$ are SFTs, then so is the dynamical direct product $X \times T$. 

\begin{definition}[Sofic shifts]
\label{def:sofic}
    A subshift $W$ is \textit{sofic} if there exists an SFT $X$ which factors onto $W$.
\end{definition}

The following elementary facts are needed; we abbreviate the proofs as they are similar to the well-known the proofs in the case where $G = \mathbb Z$ (see \cite{lind_marcus}).

\begin{prop}
\label{sofic-sft-1block-code}
    Let $X$ be an SFT, let $W$ be a sofic shift, and let $\phi : X \to W$ be a factor map. Then there exists an SFT $\tilde X$ and a topological conjugacy $\tilde \phi : \tilde X \to X$ such that the composition $\phi \circ \tilde \phi : \tilde X \to W$ is a one-block code.
\end{prop}

\begin{proof}
    Because $\phi$ is continuous and shift-commuting, there exists a large finite subset $K \subset G$ such that for each $x$, $x'\in X$ and each $g \in G$, it holds that \[
        (\sigma^g x)(K) = (\sigma^g x')(K) \implies \phi(x)_g = \phi(x')_g.
    \] Suppose that $e\in K$ and that $\P(K,X)$ specifies $X$ as an SFT. Let $\tilde \A = \P(K,X)$ be a new finite alphabet, and let $\tilde X \subset \tilde \A^G$ be the set of all points $\tilde x \in \tilde \A^G$ such that \[
        \exists x\in X, \, \forall g\in G, \,  \tilde x_g = (\sigma^g x)(K).
    \] Then $\tilde X$ is an SFT specified by patterns of shape $K^{-1}K$. The map $\tilde \phi : \tilde X \to X$ desired for the theorem is given by \[
        \tilde \phi(\tilde x)_g = (\tilde x_g)_e \in \A, \quad \forall g\in G, \, \forall \tilde x \in \tilde X.
    \]
\end{proof}

\begin{prop}
\label{elem-sft-result1}
    For any subshift $X \subset \A^G$, there is a descending family of SFTs $(X_n)_n$ such that $X = \bigcap_n X_n$.
\end{prop}

\begin{proof}
    Let $(p_n)_n$ enumerate $\{p\in \A^* : \text{$p$ does not occur in $X$}\}$, and for each $n$ let\[
        X_n = \mathcal R\big(\A^G, \{p_1, p_2, \ldots, p_n\}\big).
    \] Then $(X_n)_n$ witnesses the result.
\end{proof}

\begin{prop}
\label{elem-sft-result2}
    Let $X \subset \A^G$ be a subshift and let $X_0 \subset \A^G$ be an SFT such that $X \subset X_0$. If $(X_n)_n$ is any descending family of subshifts such that $X = \bigcap_n X_n$, then $X_n \subset X_0$ for all sufficiently large $n$.
\end{prop}

\begin{proof}
    Take $K\subset G$ to specify $X_0$ as an SFT. Note $\big(\P(K,X_n)\big)_n$ is a descending family of finite sets, and it is therefore eventually constant. In particular, we have \[
        \P(K,X_n) = \P(K,X) \subset \P(K,X_0)
    \] for all sufficiently large $n$. %The result proceeds.
\end{proof}

When $G = \Z^d$, SFTs are often reduced via conjugacy to so-called \textit{1-step} SFTs, in which the allowed patterns are specified by a $d$-hypercube of side-length 1. Such SFTs are often desired because they allow for a kind of ``surgery" of patterns. If two patterns occur in two different labellings from a 1-step SFT, and yet they agree on their 1-boundaries, then the first may be \textit{excised} and \textit{replaced} by the second. This yields a new labelling which also belongs to the 1-step SFT.  Although there is no obvious notion of 1-step SFTs when $G \neq \mathbb Z^d$, we do have the following result which allows for this sort of excision and replacement of patterns.

\begin{lemma}
\label{excision_lemma}
    Let $X \subset \A^G$ be an SFT specified by $K \subset G$, let $F \subset G$ be a finite subset, and let $x$, $y \in X$ be two points such that $x$ and $y$ agree on $\partial_{KK^{-1}} F$. Then the point $z$, defined by $z_g = y_g$ if $g \in F$ and $z_g = x_g$ if $g \notin F$, also belongs to $X$. 
\end{lemma}

\begin{proof}
    Let $g\in G$. By Lemma~\ref{int-complement}, either $Kg \subset F$ or $Kg \subset \big(\int_{KK^{-1}} F\big)^c$. In the first case, we have $(\sigma^g z)(K) = (\sigma^g y)(K)$ which is an allowed pattern in $X$. In the second case, we have $Kg \subset (F^c) \sqcup (\partial_{KK^{-1}}F)$. Since $x$ and $y$ agree on $\partial_{KK^{-1}} F$, we have $(\sigma^g z)(K) = (\sigma^g x)(K)$ which is again an allowed pattern in $X$. In either case, $(\sigma^g z)(K)$ is allowed in $X$ for every $g$, hence $z \in X$.
\end{proof}

\subsection{Entropy}

Let $X \subset \A^G$ be a nonempty subshift. Recall that for a given large finite set $F \subset G$, the number of patterns of shape $F$ that occur in any point of $X$ is $|\P(F,X)|$, which is at most $|\A|^{|F|}$.  As this grows exponentially (with respect to $|F|$), we are interested in the \textit{exponential growth rate} of $|\P(F,X)|$ as $F$ becomes very large and approaches the whole group $G$. For nonempty finite sets $F \subset G$, we let \[
    h(F,X) = \frac1{|F|}\log|\P(F,X)|.
\] If $F$, $F' \subset G$ are disjoint finite subsets, then $h(F \sqcup F', X) \leq h(F,X) + h(F', X)$. This is because $|\P(F\sqcup F',X)| \leq |\P(F,X)| \cdot |\P(F',X)|$ and \[
    \frac1{|F\sqcup F'|} = \frac1{|F|+|F'|} \leq \min \Big( \frac1{|F|},\, \frac1{|F'|} \Big).
\] 

\begin{definition}[Entropy]
\label{def:entropy}
    Let $X$ be a nonempty subshift. The \textit{(topological) entropy} of $X$ is the nonnegative real number $h(X)$ given by the limit \[
        h(X) = \lim_{n\to\infty} h(F_n, X),
    \] where $(F_n)_n$ is again the F\o lner sequence of $G$.  For the empty subshift, we adopt the convention that $h(\varnothing) = 0$.
\end{definition}

It is well-known that the limit above exists, does not depend on the choice of F\o lner sequence for $G$, and is an invariant of topological conjugacy (see \cite{kerr_li}).

For any subshift $X \subset \A^G$ and any finite subset $F\subset G$ it holds that $h(F,X) \leq \log |\A|$, and consequently $h(X) \leq \log|\A|$. More generally, if $X$ and $X'$ are subshifts such that $X \subset X'$, then $h(F,X) \leq h(F,X')$ for every finite subset $F \subset G$ and consequently $h(X) \leq h(X')$. If $X$ and $X'$ are subshifts over $\A$, then so is $X \cup X'$ and $h(X \cup X') = \max\big(h(X), h(X')\big)$.

The following proposition is a classical fact; a proof is given in \cite{kerr_li}.

\begin{prop}
    Let $G$ be a countable amenable group. If a $G$-shift $W$ is a factor of a $G$-shift $X$, then $h(W) \leq h(X)$.
\end{prop}

Frequently in this paper we refer to ``measuring" or \textit{approximating} the entropy of a subshift via a large set $F$. We give a precise definition as follows. %The precise form of this notion is as follows.

\begin{definition}
\label{def:entropy-approx}
    Let $X \subset \A^G$ be a subshift, and let $\delta > 0$. A finite subset $F \subset G$ is said to $\delta$-\textit{approximate} the entropy of $X$ if\[
        h(X) - \delta < h(F,X) < h(X) + \delta.
    \] We shall more commonly write $h(X) < h(F,X) + \delta < h(X) + 2\delta$.
\end{definition}

Infinitely many such sets exist for any $\delta$, as provided by the F\o lner sequence and the definition of $h(X)$. We introduce this notion so that we may layer invariance conditions and entropy-approximating conditions as needed.

\begin{prop}
\label{prop:entropy-and-invariance-conds}
    For finitely many choices of $i$, let $K_i \subset G$ be any finite subsets, and let $\varepsilon_i > 0$ be any positive constants. For finitely many choices of $j$, let $X_j \subset \A_j^G$ be any subshifts over any finite alphabets, and let $\delta_j > 0$ be any positive constants.  There exists a finite subset $F \subset G$ which is $(K_i,\varepsilon_i)$-invariant for every $i$, and which $\delta_j$-approximates the entropy of $X_j$ for every $j$.
\end{prop}

\begin{proof}
    Choose $F = F_n$ for sufficiently large $n$.
\end{proof}

The following theorem is an elementary generalization of a classical statement (see \cite{lind_marcus} for a proof in the case where $G = \Z$). We omit the proof here for brevity.

\begin{prop}
\label{prop:entropy-limit}
    Let $(X_n)_{n}$ be a descending family of subshifts, and let $X = \bigcap_n X_n$. Then \[
        h(X) = \lim_{n\to\infty} h(X_n).
    \]
\end{prop}

It is desirable to work with SFTs as much as possible while preserving (or, in our case, approximating) relevant dynamical quantities. We shall make frequent use of the next theorem, which we justify with several of the above results.

\begin{theorem}
\label{sft-entropy-approx}
    Let $X \subset \A^G$ be a subshift and suppose that $X_0 \subset \A^G$ is an SFT such that $X \subset X_0$. For any $\varepsilon > 0$, there exists an SFT $Z \subset \A^G$ such that $X \subset Z \subset X_0$ and $h(X) \leq h(Z) < h(X) + \varepsilon$.
\end{theorem}

\begin{proof}
    By Proposition~\ref{elem-sft-result1}, there is a descending family of SFTs $(X_n)_n$ such that $X = \bigcap_n X_n$. By Proposition~\ref{elem-sft-result2}, we have $X_n \subset X_0$ for all sufficiently large $n$. By Proposition~\ref{prop:entropy-limit}, we have $h(X) \leq h(X_n) < h(X) + \varepsilon$ for all sufficiently large $n$. Choose $Z = X_n$ for $n$ large enough to meet both conditions.
\end{proof}

%The assumption that $X_0$ is an SFT is significant. %Given subshifts $X \subset X_0$, we describe producing the shift $Z$ from the above theorem as \textit{approximating} $X$ (from above) by an SFT with $\varepsilon$-close entropy.

If $\phi : X \to W$ is a factor map of subshifts, then we have already seen that $h(W) \leq h(X)$. The ``entropy drop" or \textit{entropy gap} between $X$ and $W$ is the quantity $h(X) - h(W)$. A subsystem $X' \subset X$ induces a corresponding subsystem $\phi(X') = W' \subset W$, and later in this paper we will want a uniform bound for the entropy gap between every $X'$ and $W'$ pair. We make this idea precise in the following definition. %We put a definition to this idea now.

\begin{definition}
\label{def:max-entropy-gap}
    Suppose $\phi : X \to W$ is a factor map. The \textit{maximal entropy gap} of $\phi$ is the quantity \[
        \mathcal H(\phi) = \sup_{X'} \big(h(X') - h(\phi(X'))\big),
    \] where the supremum is taken over all subshifts $X' \subset X$. In particular, it holds that \[
        h(W) \leq h(X) \leq h(W) + \mathcal H(\phi).
    \]
\end{definition}

Recall that if $X$ and $T$ are subshifts, then the dynamical direct product $X \times T$ factors onto both $X$ and $T$ via the projection map(s) $\pi_X(x,t) = x$ and $\pi_T(x,t) = t$.

\begin{prop}
\label{lem:prod-sys-cond-entropy}
    Let $X$ and $T$ be shift spaces. The maximal entropy gap of the projection map $\pi_X : X \times T \to X$ is \[
        \mathcal H(\pi_X) = h(T).
    \]
\end{prop}

\begin{proof}
    It is classically known that $h(X \times T) = h(X) + h(T)$, in which case $h(T) = h(X \times T) - h(X) \leq \mathcal H(\pi_X)$. For the converse inequality, suppose $Z \subset X \times T$ is any subshift. Note by Definition \ref{def:product-of-shifts} that $z \in Z$ implies $z = (z^X, z^T)$, where $z^X = \pi_X(z) \in \pi_X(Z) \subset X$ and $z^T \in T$. Therefore $Z \subset \pi_X(Z) \times T$, in which case it follows that $h(Z) \leq h(\pi_X(Z)) + h(T)$. Since $Z$ was arbitrary, we have \[
        h(T) \leq \mathcal H(\pi_X) = \sup_Z \big( h(Z) - h(\pi_X(Z))\big) \leq h(T),
    \] where the supremum is taken over all subshifts $Z \subset X \times T$.
\end{proof}

A quick corollary is that when $h(T) = 0$, we have $h(Z) = h(\pi_X(Z))$ for any subsystem $Z\subset X\times T$.

\section{Tilings of amenable groups}

\subsection{Definition and encoding}

In this section we consider the notion of \textit{tilings} of $G$. The existence of tilings of $G$ with certain properties is essential in our constructions in subsequent sections. %, and state that one suitable for our purposes exists.

\begin{definition}[Quasi-tilings and exact tilings]
\label{def:quasitilings}
    A \textit{quasi-tiling} of $G$ is a pair $(\S, C)$, where $\S$ is a finite collection of finite subsets of $G$ (called the \textit{shapes} of the tiling) and $C$ is a function that assigns each shape $S \in \S$ to a subset $C(S) \subset G$, called the set of \textit{centers} or \textit{center-set} attributed to $S$. We require that $e$ is in $S$ for each $S \in \S$. The following properties are also required. \begin{enumerate}[i.]
        \item For distinct shapes $S$, $S' \in \S$, the subsets $C(S)$ and $C(S')$ are disjoint.
        \item The shapes in $\S$ are ``translate-unique", in the sense that\[
            S \neq S' \implies Sg \neq S', \quad \forall g \in G,
        \] for each $S$, $S' \in \S$. 
        \item The map $(S,c) \mapsto Sc \subset G$ defined on the domain $\{(S,c) : S \in \S \text{ and } c \in C(S)\}$ is injective.
    \end{enumerate}
    
    We may refer to both the pair $(\S,C)$ and the collection  \[
        \T = \T(\S,C) = \{Sc \subset G : S \in \S \text{ and } c \in C(S)\}
    \] as ``\textit{the} quasi-tiling." Each subset $\tau = Sc \in \T$ is called a \textit{tile}. For a quasi-tiling $\T$, we denote the union of all the tiles by $\bigcup \T$. A quasi-tiling $\T$ may not necessarily cover $G$ in the sense that $\bigcup \T = G$; nor is it necessary for any two distinct tiles $\tau$, $\tau' \in \T$ to be disjoint. However, if both of these conditions are met (that is, if $\T$ is a \textit{partition} of $G$), then $\T$ is called an \textit{exact} tiling of $G$.
\end{definition}

Ornstein and Weiss \cite{ornstein_weiss} previously constructed quasi-tilings of $G$ with good dynamical properties, and this construction has become a fundamental tool for analyzing the dynamics of $G$-actions. Downarowicz, Huczek, and Zhang \cite{tilings} sharpened this construction, showing that a countable amenable group exhibits many \textit{exact} tilings with good dynamical properties{, as we describe below} (see Theorem~\ref{downarowicz-thm}). 

A quasi-tiling $\T$ of $G$ may be encoded in symbolic form, allowing for dynamical properties to be attributed to and studied for quasi-tilings. The encoding method presented here differs from the one presented in \cite{tilings}, as we will only require exact tilings in this paper. See Remark~\ref{rem:encoding-equivalence} below for further discussion of the relation between our encoding and the encoding given in \cite{tilings}.

\begin{definition}[Encoding]
\label{def:tiling-encoding}
    Let $\S$ be a finite collection of finite shapes, and let \[
        \Sigma(\S) = \{(S,s) : s \in S \in \S\},
    \] which we view as a finite alphabet. If $\T$ is an exact tiling of $G$ over $\S$, then it corresponds to a unique point $t\in \Sigma^G$ as follows. For each $g \in G$, there is a unique tile $Sc \in \T$ containing $g$; let $s = gc^{-1} \in S$ and set $t_g = (S,s)$.
\end{definition}

{In the above definition, note that} $s$ is the ``relative position" of $g$ in the translate $Sc$ of $S$. In other words, $t$ labels each element $g$ of $G$ with both the type of shape of the tile containing $g$ and the relative position of $g$ within that tile. In particular, $g \in C(S) \iff t_g = (S,e)$.

{Note that the correspondence $\T \mapsto t \in \Sigma^G$, when regarded as a map on the set of all exact tilings of $G$ over $\S$, is injective}. {However, the correspondence is not surjective in general.} Let $\Sigma_E \subset \Sigma^G$ be the {set of all encodings of exact tilings of $G$ over $\S$.} It may be the case that \textit{no} exact tiling of $G$ over $\S$ exists, in which case $\Sigma_E = \varnothing$. In general, we have the following useful theorem.

\begin{prop}
\label{prop:tiling-SFT}
    Let $\S$ be a finite collection of finite shapes drawn from $G$. Then $\Sigma_E(\S) \subset \Sigma(\S)^G$ is an SFT.
\end{prop}

\begin{proof}
    Let $\Sigma_1$ be the set of all points $t \in \Sigma^G$ that satisfy the following local rule: for each $g \in G$, if $t_g = (S_0, s_0) \in \Sigma$ then \begin{equation}\tag{R1}
    \label{rule}
        t_{sc} = (S_0, s), \quad \forall s \in S_0,
    \end{equation} where $c = s_0^{-1}g$. It is easy to see that $\Sigma_1$ is an SFT, and from Definition \ref{def:tiling-encoding} it is immediate that $\Sigma_E \subset \Sigma_1$. 
    
    For the reverse inclusion, let $t \in \Sigma_1$ be an arbitrary point satisfying the local rule (R1) everywhere. For each $S \in \S$, let $C(S) = \{g\in G : t_g = (S,e)\}$. Then $\T = \T(\S, C)$ is a quasi-tiling. To complete the proof, it suffices to show that $\T$ is exact and encoded by $t$, since that would give $t \in \Sigma_E$ and then $\Sigma_E(\S) = \Sigma_1$. % and the  follows.
    
    Let $g \in G$, suppose $t_g = (S,s)$, and let $c = s^{-1}g$. By rule (R1) and the fact that $e\in S$, we have $t_c = t_{ec} = (S,e)$ and therefore $c \in C(S)$. Hence, $g = sc \in Sc \in \T$. This demonstrates that $\bigcup \T = G$. Next, suppose $Sc$, $S'c' \in \T$ are not disjoint and let $g \in Sc \cap S'c'$. Then $g = sc = s'c'$ for some $s \in S$ and $s' \in S'$. From $c \in C(S)$ we have $t_c = (S,e)$, and by the rule (R1) we have \[
        t_g = t_{sc} = (S,s).
    \] By identical proof we have $t_g = (S', s')$, from which it follows that $S = S'$ and $s = s'$. The latter implies that \[
        c = s^{-1}g = s'^{-1}g = c',
    \] and hence $Sc$ and $S'c'$ are the same tile. This demonstrates that $\T$ is a partition of $G$, and therefore $\T$ is an exact tiling of $G$ over $S$. Finally, we note that it is straightforward to check that $\T$ is encoded by $t$, which completes the proof. %To complete the proof, the tiling $\T$ must be encoded by the point $t$ according to Definition \ref{def:tiling-encoding}. This is elementary.
\end{proof}

\begin{remark}
\label{rem:encoding-equivalence}
    Before we move on, we note here that the encoding method presented above (Definition \ref{def:tiling-encoding}) differs from the one presented in \cite{tilings}. The encoding method in that work gives symbolic encodings for all quasi-tilings, which is not necessary for our present purposes. Indeed, the encoding in \cite{tilings} uses the alphabet $\Lambda = \S \cup \{0\}$, and a point $\lambda \in \Lambda^G$ encodes a quasi-tiling $(\S,C)$ when $\lambda_g = S \iff g \in C(S)$ and $\lambda_g = 0$ otherwise. This is a prudent encoding method for the study of general quasi-tilings, as any quasi-tiling may be encoded in this manner. Our encoding method works only for exact tilings, but is well-suited to our purposes. In fact, if one is only interested in exact tilings, then the two encodings are equivalent. Indeed, if $\Lambda_E \subset \Lambda^G$ is the collection of all encodings of exact tilings of $G$ over $\S$, then there is a topological conjugacy $\phi : \Sigma_E \to \Lambda_E$ given by $\phi(t)_g = S \iff t_g = (S,e)$ and $\phi(t)_g = 0$ otherwise.
\end{remark}
    
    Next we turn our attention to the dynamical properties of tilings, as derived from their encodings.

\begin{definition}[Dynamical tiling system]
    Let $\S$ be a finite collection of finite shapes, let $\T$ be an exact tiling of $G$ over $\S$, and let $\T$ be encoded by the point $t \in \Sigma_E(\S)$. The \textit{dynamical tiling system} generated by $\T$ is the subshift generated by $t$ in $\Sigma^G$, denoted $\Sigma_\T = \overline{\mathcal O}(t) \subset \Sigma_E$.
    
    This allows for the dynamical properties (e.g., entropy) of $\Sigma_\T$ as a subshift of $\Sigma^G$ to be ascribed to $\T$. The \textit{tiling entropy} of $\T$ is $h(\T) = h(\Sigma_\T)$, the entropy of $\Sigma_\T$ as a subshift of $\Sigma^G$.    
\end{definition}

 The tiling entropy of $\T$ is a measure of the ``complexity" of tile patterns that occur in large regions of $G$. In particular, when $\T$ has entropy zero, the number of ways to cover a large region $F \subset G$ by tiles in $\T$ grows subexponentially (with respect to $|F|$).

The following theorem is quickly deduced from the main result of Downarowicz, Huczek, and Zhang \cite{tilings}, which we state in this form for convenience. It is this result that allows us to utilize exact tilings of $G$ in this paper.

\begin{theorem}[\cite{tilings}]
\label{downarowicz-thm}
    Let $K \subset G$ be a finite subset, and let $\varepsilon > 0$. Then there exists a finite collection of finite shapes $\S$ with the following properties. \begin{enumerate}[i.]
        \item Each shape $S \in \S$ is $(K,\varepsilon)$-invariant.
        \item $K \subset S$ and $|S| > \varepsilon^{-1}$ for each shape $S \in \S$.
        \item There exists a point $t_0 \in \Sigma_E(\S)$ such that $h(\overline{\mathcal{O}}(t_0)) = 0$.
    \end{enumerate}
    The point $t_0$ encodes an exact tiling $\T_0$ of $G$ over $\S$ with tiling entropy $h(\T_0) = 0$.
\end{theorem}

\subsection{Approximating sets with tiles}

Entropy and other dynamical properties of $G$-shifts are well measured by sets with strong invariance properties (the F\o lner sequence $F_n$ provides a wealth of such sets). However, we would instead like to utilize an (appropriately selected) exact tiling $\T$ for this purpose. In this section, we build good \textit{tile approximations} of sets: finite collections of tiles $\T^* \subset \T$ attributed to large, suitably invariant subsets $F \subset G$ that are \textit{good} in the sense that the symmetric difference $F \triangle \bigcup \T^*$ is small (as a proportion of $|F|$).

\begin{definition}[Tile approximation]
\label{def:tile-approx}
    Let $F \subset G$ be a finite subset. An exact tiling $\T$ of $G$ induces two finite collections of tiles: the \textit{outer approximation} of $F$ by $\T$, denoted \[
        \T^\times(F) = \{\tau \in \T : \tau \cap F \neq \varnothing \},
    \] and the \textit{inner approximation} of $F$ by $\T$, denoted \[
        \T^\circ(F) = \{\tau \in \T : \tau \subset F\}.
    \] Denote $F^\times(\T) = \bigcup \T^\times(F)$ and $F^\circ(\T) = \bigcup \T^\circ(F)$. Observe that $F^\circ \subset F \subset F^\times$.
\end{definition}

\begin{lemma}
\label{lem:invar-cond-for-tile-approx}
    Let $\S$ be a finite collection of shapes from $G$, and let $U = \bigcup \S$. Let $\varepsilon > 0$, and choose $\delta > 0$ such that $\delta|U||UU^{-1}| < \varepsilon$. Let $F \subset G$ be a finite subset that is $(UU^{-1},\delta)$-invariant. For any exact tiling $\T$ of $G$ over $\S$, the following statements hold: \begin{enumerate}[i.]
        \item $|F^\times(\T) \setminus F^\circ(\T)| < \varepsilon|F|$,
        \item $(1-\varepsilon)|F| < |F^\circ(\T)| \leq |F|$, and
        \item $|F| \leq |F^\times(\T)| < (1+\varepsilon)|F|$.
    \end{enumerate} 
\end{lemma}

\begin{proof}
    First, we observe that each tile $\tau \in \T$ is contained in a translate $Ug$ for some $g\in G$; indeed, we have $\tau = Sc$ for some $S \in \S$ and $c \in C(S) \subset G$, then $S \subset U$ implies $\tau \subset Uc$. This fact also gives that $|\tau| \leq |U|$ for every tile $\tau \in \T$.

    We claim that every tile $\tau \in \T^\times(F) \setminus \T^\circ(F)$ intersects $\partial_{UU^{-1}}F$. To establish the claim, we first note that for each such tile $\tau$ it holds that $\tau \cap F \neq \varnothing$ and $\tau \not\subset F$. So let $f \in \tau \cap F$, and note that $f\in \tau \subset Ug$ for some $g\in G$. From $\tau \not\subset F$ we also have $Ug \not \subset F$. By Lemma~\ref{int-complement} we have  $f\in Ug \subset \big(\int_{UU^{-1}} F\big)^c$, and hence $f \in F \setminus\big(\int_{UU^{-1}} F\big) = \partial_{UU^{-1}} F$, which establishes our claim. 
    
    By the claim in the previous paragraph, there is a map $\gamma : \T^\times \setminus \T^\circ \to \partial_{UU^{-1}} F$ with the property that $\gamma(\tau) \in \tau$ for each $\tau$. Observe that $\gamma$ is injective, as distinct tiles are disjoint, and therefore $|\T^\times \setminus \T^\circ| \leq |\partial_{UU^{-1}} F|$. We also have that \[
        |\partial_{UU^{-1}}F| < \delta|UU^{-1}||F|,
    \] by the invariance hypothesis on $F$ and Lemma~\ref{small-bdry}. Then %The following estimates then proceed. 
    \begin{align*}
        |F^\times \setminus F^\circ| &= \sum_{\tau \in \T^\times \setminus \T^\circ} |\tau|\\
        &\leq |\T^\times \setminus \T^\circ||U|\\
        &\leq |\partial_{UU^{-1}} F||U|\\
        &< \delta|U||UU^{-1}||F|\\
        &< \varepsilon|F|.
    \end{align*}
    This establishes statement (\textit{i.}). The remaining two statements are easy to check using  $F^\times = F^\circ \sqcup (F^\times \setminus F^\circ)$,  $F^\circ \subset F \subset F^\times$, and statement (\textit{i.}).
\end{proof}

One more notion is necessary to develop before moving on from tilings: the \textit{frame} of a given subset with respect to a given tiling.

\begin{definition} [Frame of a tiling]
\label{def:frame}
    Let $F$, $K \subset G$, and let $\T$ be an exact tiling of $G$. The inner $(\T,K)$-\textit{frame} of $F$ is the subset \[
        \operatorname{fr}_{\T,K}(F) = \bigcup_{\tau} \partial_K(\tau),
    \] where the union ranges over all $\tau \in \T^\circ(F)$. See Figure \ref{fig:frame} for an illustration.
\end{definition}
    \begin{figure}[hbt]
        \centering
        \begin{subfigure}[h]{.4\textwidth}
            \centering
            \includegraphics[width=\textwidth]{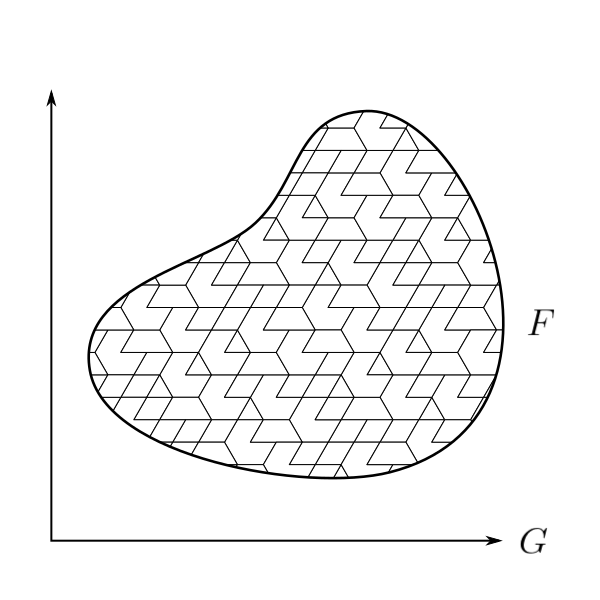}
            \caption{A hypothetical region $F \subset G$ with illustrated tiling $\T$.}
        \end{subfigure}
        \hfill
        \begin{subfigure}[h]{.4\textwidth}
            \centering
            \includegraphics[width=\textwidth]{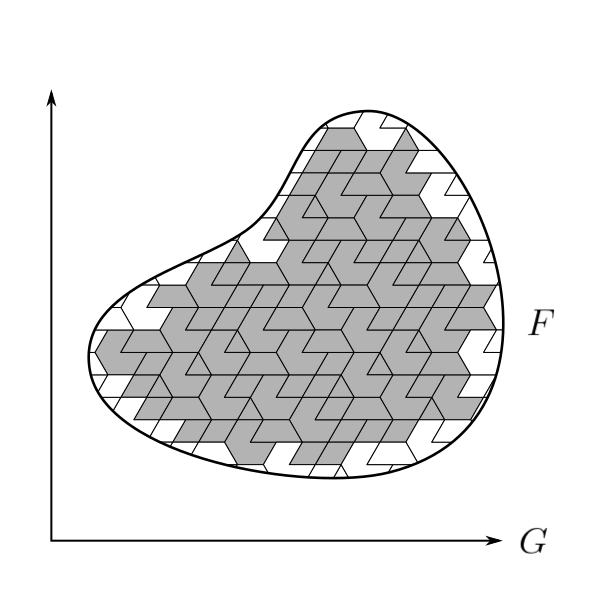}
            \caption{The $\T$-interior of $F$ is shaded.}
        \end{subfigure}
        
        \begin{subfigure}[h]{.6\textwidth}
            \centering
            \includegraphics[width=\textwidth]{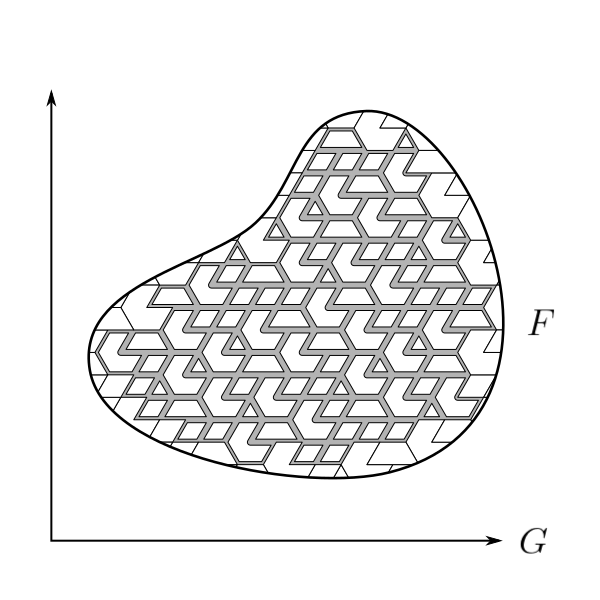}
            \caption{The inner frame of $F$ (with respect to $\T$, $K$) is shaded. The $K$-boundary of each tile inside $F$ is taken.}
        \end{subfigure}
        \caption{A sketch of the construction of $\operatorname{fr}_{\T,K}(F)$}
    \label{fig:frame}
    \end{figure}

\section{Results for SFTs}

Having discussed everything about tilings relevant for our purposes, we are now ready to begin discussing our main results. In this section we present our results for SFTs, and in the following section we turn our attention to sofic shifts.

\begin{theorem}
\label{main-result-sft}
    Let $G$ be a countable amenable group, and let $X$ be a $G$-SFT such that $h(X) > 0$. Then \[
        \{ h(Y) : Y \subset X \text{ and $Y$ is an SFT}\}
    \] is dense in $[0, h(X)]$.
\end{theorem}

Before we begin the proof, let us give a short outline of the main ideas. % very lengthy proof which follows. 
The broad strokes of this proof come from Desai \cite{desai}, whose argument in the case where $G = \mathbb Z^d$ we are able to extend to the case where $G$ is an arbitrary countable amenable group. This is possible by utilizing the exact tilings of $G$ constructed by Downarowicz, Huczek and Zhang \cite{tilings}.

Given an arbitrary $\varepsilon > 0$, we produce a family of SFT subshifts of $X$ whose entropies are $2\varepsilon$-dense in $[0,h(X)]$. We accomplish this by first selecting an exact, zero entropy tiling $\T_0$ of $G$ with suitably large, invariant tiles. Then we build subshifts with strongly controlled entropies inside the product system $Z_0 = X \times \Sigma_0$, where $\Sigma_0$ is the dynamical tiling system generated by $\T_0$.

%To construct these subshifts, we consider a certain subcollection of patterns which occur in $Z_0$: that of ``blocks" whose shapes come from $\T_0$ (or a shifted copy thereof), and for which the $\Sigma$-layer of the block is labelled properly according to the $\Sigma$-encoding of $\T_0$. Such ``tile-aligned" blocks, if one occurs inside a larger pattern, must be upon one of the tiles from $\T_0$ (or a shifted copy thereof). This controlled ``frequency" of block occurrence, coupled with the charitable invariance of the tiles themselves, implies that the aligned blocks contain very little ``information"; they contribute very little entropy to the parent subshift $Z_0$. \textcolor{blue}{I think we should revise this paragraph.}

To construct these subshifts from $Z_0$, we control which patterns in the $X$ layer can appear in the ``interior" of the tiles in the $\Sigma_0$ layer. 
We are able to finely comb away entropy from $Z_0$ by forbidding these patterns one at a time. This process generates a descending family of subsystems for which the entropy drop between consecutive subshifts is less than $\varepsilon$. After enough such patterns have been forbidden, the overall entropy is less than $\varepsilon$. This collection of subshifts therefore has entropies that are $\varepsilon$-dense in $[0, h(Z_0)]$. Then we project the subshifts into $X$ and utilize Theorem~\ref{sft-entropy-approx} to produce SFTs subsystems of $X$ with entropies that are $2\varepsilon$-dense in $[0,h(X)]$. %and approximate the images by SFTs (using Theorem~\ref{sft-entropy-approx}) to complete the proof.

\begin{proof}
    Let  $X\subset \A^G$ be an SFT such that $h(X) >0$, let $K \subset G$ be a large finite subset such that $\P(K,X)$ specifies $X$ as an SFT, and let $\varepsilon$ be any constant such that $0 < \varepsilon < h(X)$. Choose $\delta > 0$ such that \[
        2\delta + \delta \log 2 + 2 \delta \log |\A| < \varepsilon.
    \] By Theorem~\ref{downarowicz-thm}, there exists a finite collection $\S$ of finite subsets of $G$ with the following properties. \begin{enumerate}[i.]
        \item Each shape $S\in \S$ is $(KK^{-1}, \eta)$-invariant, where $\eta > 0$ is a constant such that $\eta|KK^{-1}| < \delta$. By Lemma~\ref{small-bdry}, this implies that $|\partial_{KK^{-1}} S| < \delta|S|$ for each shape $S \in \S$.
        \item $KK^{-1} \subset S$ and $|S| > \delta^{-1}$ for each $S \in \S$.
        \item There is a point $t_0 \in \Sigma_E(\S)$ such that $h(\overline{\mathcal{O}}(t_0)) = 0$. Consequently, $t_0$ encodes an exact tiling $\T_0$ of $G$ over $\S$ with tiling entropy zero.
    \end{enumerate}
    
    For the remainder of this proof, these are all fixed. We shall abbreviate $\partial F = \partial_{KK^{-1}}F$ for any finite subset $F \subset G$. For a pattern $p$ on $F$, we take $\partial p$ to mean $p(\partial F)$ and call this the \textit{border} of $p$ (with respect to $KK^{-1}$).
    
    Let $\Sigma_0 = \overline{\mathcal O}(t_0) \subset \Sigma_E$ be the dynamical tiling system generated by the tiling $\T_0$, which has entropy zero. Of central importance to this proof is the product system $X \times \Sigma_0$, which factors onto $X$ via the projection map $\pi : X \times \Sigma_0 \to X$ given by $\pi(x,t) = x$ for each $(x,t) \in X \times \Sigma_0$. Let us establish some terminology for certain patterns of interest which occur in this system.
    
    Given a shape $S \in \S$, we shall refer to a pattern $b = (b^X, b^\T) \in \P(S, X \times \Sigma_0)$ as a \textit{block} (to distinguish from patterns of any general shape). If a block $b \in (\A \times \Sigma)^S$ satisfies $b^\T_s = (S,s)$ for every $s \in S$, then we shall say $b$ is \textit{aligned}. See Figure \ref{fig:aligned} for an illustration of the aligned property. 
    
    \begin{figure}[hbt]
    	\centering
    	\includegraphics[width=.8\textwidth]{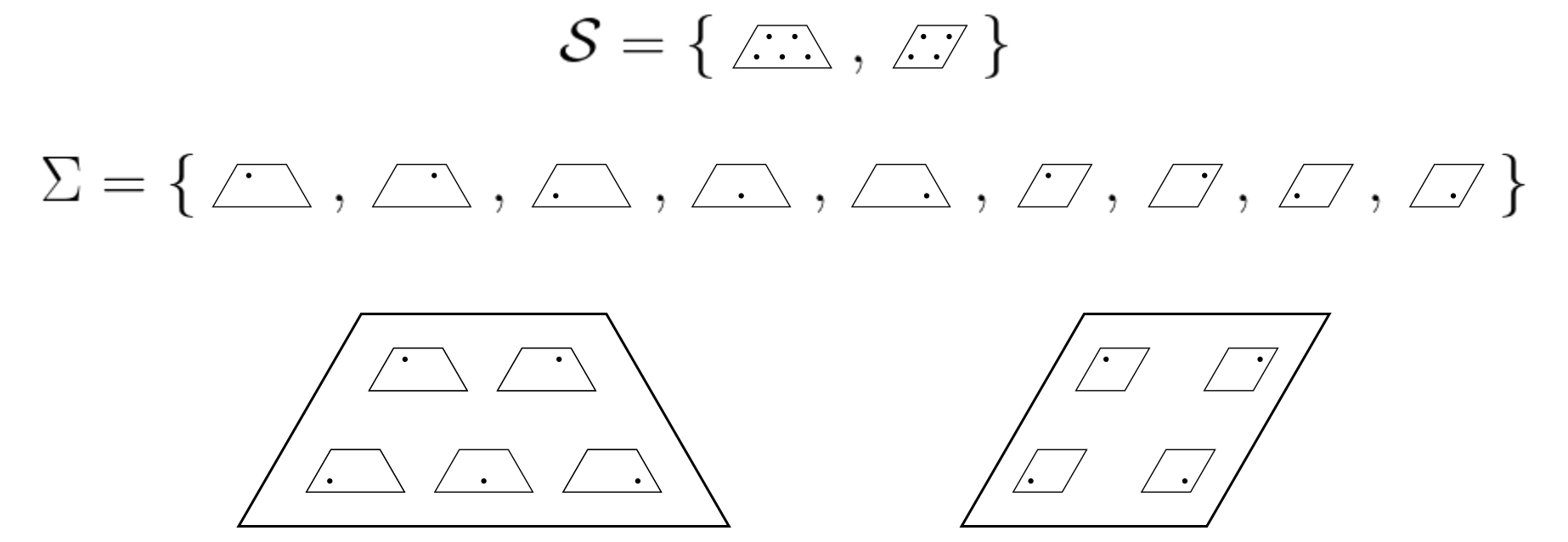}
    	\caption{A hypothetical collection of shapes $\S$, the appropriate alphabet $\Sigma = \Sigma(\S)$, and (the $\T$-layer of) two aligned blocks are pictured. Each point of each block is labelled with the correct shape type and relative displacement within that shape.} %\textcolor{blue}{To be clear, this figure only displays the $\T$ component of a block, right?}}
    \label{fig:aligned}
    \end{figure}

    For a subshift $Z \subset X \times \Sigma_0$, we denote the subcollection of aligned blocks of shape $S$ that occur in $Z$ by \[
        \P^a(S,Z) \subset \P(S,Z) \subset (\A \times \Sigma)^S,
    \] where the superscript $a$ identifies the subcollection. Given a shape $S \in \S$ and an aligned block $b$ of shape $S$, consider the border $\partial b \in (\A \times \Sigma)^{\partial S}$. We are interested in the number of ways that the border $\partial b$ may be extended to all of $S$ - that is, the number of allowed (and in particular, aligned) \textit{interiors} for $S$ which agree with $\partial b$ on the boundary $\partial S$. For a subshift $Z \subset X \times \Sigma_0$, we denote this collection by \[
        \ints^a(\partial b, Z) = \{b' \in \P^a(S, Z) : \partial b' = \partial b\}.
    \] We shall extend all the same terminology described above (blocks, aligned blocks, borders, interiors) to tiles $\tau = Sc \in \T_0$, as there is a bijection between $\P(S,Z)$ and $\P(\tau, Z) = \P(Sc, Z)$. For a given tile $\tau \in \T_0$, a block $b \in (\A \times \Sigma)^\tau$ is aligned if $b^\T_{sc} = (S,s)$ for each $sc \in Sc = \tau$. The subcollection of aligned blocks of shape $\tau$ occurring in a shift $Z \subset X \times \Sigma_0$ is denoted $\P^a(\tau, Z)$. Given a border $\partial b \in (\A \times \Sigma)^{\partial \tau}$, the collection of aligned blocks of shape $\tau$ occurring in $Z$ agreeing with $\partial b$ on $\partial \tau$ is also denoted $\ints^a(\partial b, Z) \subset \P^a(\tau, Z)$.
    
    %%%%%%%%%%%%%%%%%%%%%%%%%%%%%%%%
    
    For the theorem, we shall inductively construct a descending family of subshifts $(Z_n)_{n}$ of $X \times \Sigma_0$ as follows. Begin with $Z_0 = X \times \Sigma_0$, then assume $Z_n$ has been constructed for $n \geq 0$. If there exists a shape $S_n \in \S$ and an aligned block $\beta_n \in \P^a(S_n, Z_n)$ such that \[
        |\ints^a(\partial \beta_n, Z_n)| > 1,
    \] then let $Z_{n+1} = Z_n \setminus \beta_n$. If no such block exists on any shape $S\in\S$, then $Z_n$ is the final subshift in the chain and the chain is finite in length.
    
    Let us first argue that in fact, the chain \textit{must} be finite in length. For each $n \geq 0$ we have $Z_{n+1} \subset Z_n$, in which case $\P^a(S,Z_{n+1}) \subset \P^a(S,Z_n)$ for every shape $S \in \S$. Moreover, for the distinguished shape $S_n$ (the shape of the forbidden block $\beta_n$), it holds that $\P^a(S_n, Z_{n+1}) \sqcup \{\beta_n\} \subset \P^a(S_n, Z_n)$. This implies that \[
        \sum_{S \in \S} |\P^a(S, Z_n)|
    \] strictly decreases with $n$. There is no infinite strictly decreasing sequence of positive integers, hence the descending chain must be finite in length. Let $N \geq 0$ be the index of the terminal subshift, and note by construction that the shift $Z_N$ satisfies \[
        |\ints^a(\partial b, Z_N)| = 1
    \] for every aligned block $b \in \P^a(S, Z_N)$ on any shape $S \in \S$.
    
    Most of the rest of the proof aims to establish the following two statements: \begin{equation}\tag{U1}
    \label{claim:small-entropy-gap}
        h(Z_{n+1}) \leq h(Z_n) < h(Z_{n+1}) + \varepsilon \quad \text{for each } n < N, \text{ and}
    \end{equation} \begin{equation}\tag{U2}
    \label{claim:small-terminal-entropy}
        h(Z_N) < \varepsilon.
    \end{equation}
    
    To begin, let $F \subset G$ be a finite subset satisfying the following two conditions: \begin{enumerate}[(F1)]
        \item $F$ is $(UU^{-1},\vartheta)$-invariant, where $U = \bigcup \S$ and $\vartheta$ is a positive constant such that $\vartheta|U||UU^{-1}| < \delta$. Note this implies that $F$ may be well approximated by tiles from any exact tiling of $G$ over $\S$, in the sense of Lemma~\ref{lem:invar-cond-for-tile-approx}.
        \item $F$ is large enough to $\delta$-approximate (Definition \ref{def:entropy-approx}) the entropy of $\Sigma_0$ and $Z_n$ for every $n\leq N$. This implies in particular that $h(F,\Sigma_0) < \delta$.
    \end{enumerate}
    Such a set exists by Proposition~\ref{prop:entropy-and-invariance-conds}. We fix $F$ for the remainder of this proof. 
    
    Now for each $n \leq N$, we claim that \begin{gather*}
        |\P(F,Z_n)| \geq \sum_t \sum_f \prod_\tau |\ints^a(f(\partial\tau),Z_n)|, \quad \text{and} \tag{E$1$} \label{claim:counting-formula-1}\\
        |\P(F,Z_n)| \leq |\A|^{\delta|F|} \cdot \sum_t \sum_f \prod_\tau |\ints^a(f(\partial\tau),Z_n)|, \tag{E$2$} \label{claim:counting-formula-2}
    \end{gather*}
    where the indices $t$, $f$, and $\tau$ are as follows.
    The variable $t$ ranges over $\P(F,\Sigma_0)$, and therefore $t$ is the restriction to $F$ of an encoding of an exact, zero entropy tiling $\T_t$ of $G$ over $\S$. 
    %\textcolor{blue}{Do we really want $t \in \P(F,Z_0)$, or should it be $\P(F,\Sigma_0)$?} 
    The variable $f$ ranges over all $(\A\times\Sigma)$-labellings of the $(\T_t,\, KK^{-1})$-frame of $F$ (Definition \ref{def:frame}) that are allowed in $Z_0$ and for which $f^\T$ agrees with $t$. Lastly, the variable $\tau$ ranges over the tiles in $\T^\circ_t(F)$.
    
    To begin the argument towards the claims (\ref{claim:counting-formula-1}) and (\ref{claim:counting-formula-2}), let $n \leq N$ be arbitrary. To count patterns $p \in \P(F,Z_n)$, write $p = (p^X, p^\T)$ and sum over all possible labellings in the tiling component. We have \begin{equation}
    \label{eqn:1}
        |\P(F, Z_n)| = \sum_t |\{p \in \P(F,Z_n) : p^\T = t\}|,    
    \end{equation}
    where the sum ranges over all $t \in \P(F,\Sigma_0)$. This is valid because $Z_n \subset Z_0 = X \times \Sigma_0$, hence any $z = (z^X, z^\T) \in Z_n$ must have $z^\T \in \Sigma_0$.
    
    Next, let $t \in \P(F,\Sigma_0)$ be fixed. The pattern $t$ extends to/encodes an exact, zero entropy tiling $\T_t$ of $G$ over $\S$, possibly distinct\footnote{As $\Sigma_0$ is generated by $\T_0$, one may take $\T_t$ to be a translation of $\T_0$ that agrees with $t$ on $F$.} from the original selected tiling $\T_0$. 
    
    Recall that $F^\circ(\T_t) = \bigcup \T_t^\circ(F) \subset F$ is the inner tile approximation of $F$ by the tiling $\T_t$ (Definition \ref{def:tile-approx}). Recall also that the $(\T_t,\, KK^{-1})$-frame of $F$ is the subset $\bigcup_\tau \partial \tau$ where the union is taken over all $\tau \in \T_t^\circ(F)$ (Definition \ref{def:frame}). Since $K$ is fixed for this proof, we shall abbreviate the frame as $\operatorname{fr}_t(F)$. From Equation ($\ref{eqn:1}$), we now split over all allowed labellings of $\operatorname{fr}_t(F)$. We have \begin{equation}
        \label{eqn:2}
        |\P(F,Z_n)| = \sum_t \sum_f |\{ p \in \P(F,Z_n) : p^\T = t \text{ and } p(\operatorname{fr}_t(F)) = f\}|,
    \end{equation}
    where the first sum is taken over all $t \in \P(F, \Sigma_0)$, and the second sum is taken over all $f \in \P(\operatorname{fr}_t(F), Z_0)$ for which $f^\T$ agrees with $t$. %One could restrict $f$ to only the frames allowed in $Z_n$, but we may use the larger set $Z_0$ as the summands which correspond to patterns from $Z_0 \setminus Z_n$ will vanish. This is preferred as well, so that the range of the sum does not depend on $n$.
    
    We have the pattern $t \in \P(F,\Sigma_0)$ fixed from before; next we fix a frame pattern $f \in \P(\operatorname{fr}_t(F), Z_0)$ such that $f^{\T}$ agrees with $t$. We wish to count the number of patterns $p \in \P(F,Z_n)$ such that $p^\T = t$ and $p(\operatorname{fr}_t(F)) = f$. Let this collection be denoted by $D_n = D(t,f;\, Z_n) \subset \P(F,Z_n)$. Observe that each $D_n$ is finite and $D_{n+1} \subset D_n$ for each $n < N$. Consider the map $\gamma : D_0 \to \prod_\tau \P(\tau, Z_0)$ given by $\gamma(p) = (p(\tau))_{\tau}$, which sends a pattern $p \in D_0$ to a vector of blocks indexed by $\T_t^\circ(F)$. We claim the map $\gamma$ is at most $|\A|^{\delta|F|}$-to-1, and for each $n \leq N$ we have \begin{equation}
    \label{claim:gamma-set-equation}
        \gamma(D_n) = \prod_\tau \ints^a(f(\partial\tau), Z_n).
    \end{equation} Together, these claims will provide a bound for $|D_n|$ from above and below, which combine with Equation (\ref{eqn:2}) to yield the claims (\ref{claim:counting-formula-1}) and (\ref{claim:counting-formula-2}).
    
    First we argue that $\gamma$ is at most $|\A|^{\delta|F|}$-to-1. This is where we first invoke the invariance of $F$. Suppose $(b_\tau)_\tau \in \prod_\tau \P(\tau, Z_0)$ is a fixed vector of blocks. If $p \in D_0$ is a pattern such that $\gamma(p) = (b_\tau)_\tau$, then $p^\T$ is determined by $t$ and $p(\tau) = b_\tau$ for each tile $\tau \in \T_t^\circ(F)$. Therefore, $p$ is uniquely determined by $p^X(F\setminus F_t^\circ)$, hence $|\gamma^{-1}(b_\tau)_\tau| \leq |\A|^{|F\setminus F_t^\circ|}$. By property (F1) of the set $F$ and by Lemma~\ref{lem:invar-cond-for-tile-approx} we have $|F\setminus F_t^\circ| < \delta|F|$, and thus the map $\gamma$ is at most $|\A|^{\delta|F|}$-to-1.
    
    Next, we shall prove the set equality (\ref{claim:gamma-set-equation}). Let $n \leq N$, and let $p \in D_n$. For each tile $\tau \in \T_t^\circ(F)$, the block $p(\tau) \in \P(\tau, Z_n) \subset (\A \times \Sigma)^\tau$ is aligned; this is because $p^\T = t$ and $t$ encodes the tiling $\T_t$ itself. Moreover, $p$ agrees with $f$ on $\operatorname{fr}_t(F)$ by assumption that $p \in D_n = D(t,f;Z_n)$, in which case $p(\partial\tau) = f(\partial\tau)$ for each tile $\tau \in \T_t^\circ(F)$. This demonstrates that \[
        \gamma(D_n) \subset \prod_\tau \ints^a(f(\partial\tau), Z_n).
    \] We shall prove the reverse inclusion by induction on $n$. For the $n = 0$ case, let $(b_\tau)_{\tau}$ be a vector of blocks such that $b_\tau \in \P^a(\tau, Z_0)$ and $\partial b_\tau = f(\partial \tau)$ for each $\tau \in \T_t^\circ(F)$. To construct a $\gamma$-preimage of $(b_\tau)_\tau$ in $D_0$, begin with a point $x \in X$ such that $x(\operatorname{fr}_t(F)) = f^X$. Such a point exists because $f$ occurs in some point of $Z_0 = X \times \Sigma_0$. Note that \[
        x(\partial\tau) = f^X(\partial\tau) = \partial b_\tau^X
    \] for each $\tau \in \T_t^\circ(F)$, because $\partial \tau \subset \operatorname{fr}_t(F)$ for each $\tau$. Moreover, for each $\tau$ it holds that the block $b_\tau^X$ occurs in a point of $X$, as each block $b_\tau = (b_\tau^X, b_\tau^\T)$ occurs in a point of $Z_0 = X \times \Sigma_0$. Because $X$ is an SFT specified by patterns of shape $K$, we may repeatedly apply Lemma~\ref{excision_lemma} to excise the block $x(\tau)$ and replace it with $b_\tau^X$ for every $\tau \in \T_t^\circ$. Every tile is disjoint, so the order in which the blocks are replaced does not matter. After at most finitely many steps, we obtain a new point $x' \in X$ such that $x'(\operatorname{fr}_t(F)) = f^X$ and $x'(\tau) = b_\tau^X$ for each $\tau \in \T_t^\circ$.
    
    Recall that the point $t \in \Sigma_0$ is fixed from before. The point $(x', t) \in X \times \Sigma_0$ is therefore allowed in $Z_0 = X \times \Sigma_0$. Let $p = (x',t)(F) \in \P(F,Z_0)$. We have that $p^\T = t$ and $p(\operatorname{fr}_t(F)) = f$ by the selection of $x'$. This implies that $p \in D_0$. It also holds that $p(\tau) = b_\tau$ for each $\tau \in \T_t^\circ(F)$, as $t$ itself encodes the tiling $\T_t$ from which the tiles $\tau \in \T_t^\circ(F)$ are drawn (and each block $b_\tau$ is aligned, by assumption). We then finally have $\gamma(p) = (b_\tau)_\tau$, which settles the case $n = 0$.
    
    Now suppose the set equality (\ref{claim:gamma-set-equation}) holds for some fixed $n < N$, and let $(b_\tau)_\tau \in \prod_\tau \ints^a(f(\partial\tau), Z_{n+1})$. From the inclusion $Z_{n+1} \subset Z_n$ and the inductive hypothesis, it follows there is a pattern $p \in D_n$ such that $\gamma(p) = (b_\tau)_\tau$. Suppose $p = (x,t)(F)$ for some $(x,t) \in Z_n$ (by induction, $t$ is the point fixed from before). We need to modify $p$ only slightly to find a $\gamma$-preimage of $(b_\tau)_\tau$ which occurs in $Z_{n+1}$ (and hence belongs to $D_{n+1}$).
    
    Consider the block $\beta_n$ determined at the beginning of this proof, which is forbidden in the subshift $Z_{n+1}$. If $\beta_n$ occurs anywhere in the point $(x, t)$, then (by the assumption that $\beta_n$ is aligned) it must occur on a tile\footnote{This sentence is the reason why we consider the product system $X \times \Sigma_0$ rather than working in $X$ directly; the extra information on the tiling layer of a labelling $(x,t)$ allows us to control where an aligned block may occur within $(x,t)$.} $\tau \in \T_t$. It does \textit{not} occur on any of the tiles from $\T_t^\circ(F)$, because for each tile $\tau \in \T_t^\circ(F)$ we have $(x,t)(\tau) = b_\tau$ which is allowed in $Z_{n+1}$ by assumption.

    Yet, $\beta_n$ may occur in $(x, t)$ outside of $F_t^\circ$. By the construction of $Z_{n+1}$, we have \[
        |\ints^a(\partial\beta_n, Z_n)| > 1,
    \] and therefore there is an aligned block $\tilde b$ which occurs in $Z_n$ such that $\tilde b \neq \beta_n$ and $\partial \tilde b = \partial \beta_n$. Apply Lemma~\ref{excision_lemma} at most countably many times\footnote{If $x_k$ is the point constructed after $k$ excisions, then $(x_k)_k$ is a Cauchy sequence in $X$ and therefore has a limit $x_\infty \in X$, the desired point with infinitely many excisions applied. The order in which the blocks are replaced does not matter because the tiles are all disjoint.} to excise $\beta_n^X$ wherever it may occur in $x$, replacing it with $\tilde b^X$. This yields a new point $x' \in X$.

    Then $(x', t)\in Z_0$ also belongs to $Z_{n+1}$. It was already the case that none of the blocks $\beta_0, \ldots, \beta_{n-1}$ could occur anywhere in $(x,t)$ by the assumption $(x,t) \in Z_n$, and now neither does $\beta_n$ occur anywhere in $(x',t)$. The pattern $p' = (x', t)(F)$ may be distinct from $p = (x, t)(F)$ (the labelling may change on $F \setminus F_t^\circ$), but we did not replace any of the blocks within $F_t^\circ$. We still have $p'(\tau) = b_\tau$ for each tile $\tau \in \T_t^\circ(F)$, and hence $p' \in D_{n+1}$ and $\gamma(p') = (b_\tau)_\tau$.
    
    This completes the induction, and we conclude that the set equality (\ref{claim:gamma-set-equation}) holds for each $n \leq N$. From this equality and the fact that $\gamma$ is at most $|\A|^{\delta|F}$-to-1, we obtain \begin{equation*}
        \prod_\tau |\ints^a(f(\partial\tau),Z_n)| \leq |D(t,f\, ;Z_n)| \leq |\A|^{\delta|F|} \cdot \prod_\tau |\ints^a(f(\partial\tau),Z_n)|.
    \end{equation*} 
    Notice that the above inequalites hold for each fixed $t \in \P(F,\Sigma_0)$, each fixed $f \in \P(\operatorname{fr}_t(F), X \times \Sigma_0)$ such that $f^\T = t(\operatorname{fr}_t(F))$, and each $n \leq N$. From these inequalities and Equation (\ref{eqn:2}), we conclude that (\ref{claim:counting-formula-1}) and (\ref{claim:counting-formula-2}) hold, i.e., 
    \begin{gather*}
        |\P(F,Z_n)| \geq \sum_t \sum_f \prod_\tau |\ints^a(f(\partial\tau),Z_n)|, \quad \text{and} \tag{E$1$} \\
        |\P(F,Z_n)| \leq |\A|^{\delta|F|} \cdot \sum_t \sum_f \prod_\tau |\ints^a(f(\partial\tau),Z_n)|, \tag{E$2$}
    \end{gather*}    
    where the first sum is taken over all $t \in \P(F,\Sigma_0)$, the second sum over all $f \in \P(\operatorname{fr}_t(F), Z_0)$ for which $f^\T$ agrees with $t$, and the product over all $\tau \in \T_t^\circ(F)$.
    
    Property (F2) of the set $F$ implies that $h(F,\Sigma_0) < \delta$, in which case $|\P(F,Z_0)| < e^{\delta|F|}$. Consequently, the variable $t$ in (\ref{claim:counting-formula-1}) and (\ref{claim:counting-formula-2}) ranges over fewer than $e^{\delta|F|}$ terms. Moreover, by the selection of $\S$, we have $|\partial\tau| < \delta|\tau|$ for each $\tau \in \T_t^\circ(F)$, in which case it follows that \[
        |\operatorname{fr}_t(F)| =  \Big|\bigcup_\tau \partial\tau \Big| = \sum_\tau |\partial\tau| < \sum_\tau \delta|\tau| = \delta \Big|\bigcup_\tau \tau \Big| = \delta |F^\circ_t| \leq \delta|F|.
    \] Here we have used that distinct tiles from $\T_t$ are disjoint. From this estimate, we deduce that there are fewer than $|\A|^{\delta|F|}$ labellings of the frame of $F$ that agree with a fixed $t$ on the $\T$-layer\footnote{It is important to specify that the $\T$-layer is fixed, else the number is at most $|\A\times\Sigma|^{\delta|F|}$, which cannot be bounded because $\delta$ must be selected before $\Sigma$ is constructed.}. Consequently, the variable $f$ in (\ref{claim:counting-formula-1}) and (\ref{claim:counting-formula-2}) ranges over fewer than $|\A|^{\delta|F|}$ terms. Observe also that the size of $\T_t^\circ(F)$ as a collection is small compared to $F$. Indeed, $|S| > \delta^{-1}$ for each shape $S \in \S$, in which case \[
        |F| \geq |F^\circ_t| = \sum_\tau |\tau| \geq |\T_t^\circ(F)| \cdot \big( \min_{S\in \S} |S| \big) > |\T_t^\circ(F)| \delta^{-1},
    \] and therefore $|\T_t^\circ(F)| < \delta|F|$. Consequently, the variable $\tau$ in (\ref{claim:counting-formula-1}) and (\ref{claim:counting-formula-2}) ranges over fewer than $\delta|F|$ terms.
    
    Before returning to (\ref{claim:small-entropy-gap}) and (\ref{claim:small-terminal-entropy}), one more estimate is necessary. For each $n < N$, each shape $S\in \S$, and each aligned block $b \in \P^a(S, Z_n)$, we claim that \begin{equation}
    \label{claim:inequality}
        |\ints^a(\partial b, Z_n)| \leq 2\, |\ints^a(\partial b, Z_{n+1})|.
    \end{equation} Let $S \in \S$, and let $b \in \P^a(S, Z_n)$ be an aligned block occurring in $Z_n$, distinct from the forbidden block $\beta_n$. Say $b = (x,t)(S)$ for some $(x,t) \in Z_n$. We claim that $b$ occurs in a point of $Z_{n+1}$. Note that $t$ extends to/encodes an exact, zero entropy tiling $\T_t$ of $G$ over $\S$, and note that $S$ is a tile of $\T_t$. This is because $b$ is aligned by assumption, in which case $t_e = b^\T_e = (S,e)$, and therefore $e \in C_t(S)$.
    
    Suppose the forbidden block $\beta_n$ occurs anywhere in $(x,t)$. Because $\beta_n$ is aligned, it must occur on a tile $\tau \in \T_t$. It does \textit{not} occur on $S$, because $b \neq \beta_n$. By the assumption that $|\ints^a(\partial\beta_n, Z_n)| > 1$, we know there is an aligned block $\tilde b_n \neq \beta_n$ that occurs in $Z_n$ such that $\partial \tilde b_n = \partial \beta_n$.
    
    Recall $X$ is an SFT specified by patterns of shape $K$, and $\tilde b_n^X$ is allowed in $X$. Again we may apply Lemma~\ref{excision_lemma} at most countably many times, excising $\beta_n^X$ wherever it may occur in $x$ and replacing it with $\tilde b_n^X$. At the end we receive a new point $x' \in X$, within which $\beta_n^X$ does not occur. Then $(x',t)$ is allowed in $Z_{n+1}$ and $(x',t)(S) = b$, hence $b \in \P^a(S,Z_{n+1})$. 
    
    The conclusion is that $\beta_n$ is the only aligned block lost from $Z_n$ to $Z_{n+1}$. For each $b \in \P^a(S,Z_n)$, we have either $\ints^a(\partial b, Z_n) = \ints^a(\partial b, Z_{n+1})$ or $\ints^a(\partial b, Z_n) = \ints^a(\partial b, Z_{n+1})\sqcup \{\beta_n\}$. If two positive integers differ by at most 1 then their ratio is at most 2, hence the inequality (\ref{claim:inequality}) follows.
    
    Finally, we shall use the estimates (\ref{claim:counting-formula-1}) and (\ref{claim:counting-formula-2}) to argue for the ultimate claims (\ref{claim:small-entropy-gap}) and (\ref{claim:small-terminal-entropy}) made before. For the first, consider a fixed $n < N$. It is clear that $h(Z_{n+1}) \leq h(Z_n)$ by inclusion. For the second inequality in (\ref{claim:small-entropy-gap}), we have %$F \subset G$ as selected above, 
    \begin{align*}
        |\P(F, Z_n)| &\leq |\A|^{\delta|F|} \cdot \sum_t\sum_f \prod_\tau |\ints^a(f(\partial\tau), Z_n)| \\
        &\leq |\A|^{\delta|F|} \cdot \sum_t\sum_f \prod_\tau 2\, |\ints^a(f(\partial\tau), Z_{n+1})|\\
        &< |\A|^{\delta|F|} \cdot 2^{\delta|F|}\cdot\sum_t\sum_f \prod_\tau|\ints^a(f(\partial\tau), Z_{n+1})|\\
        &\leq |\A|^{\delta|F|} \cdot 2^{\delta|F|}\cdot |\P(F, Z_{n+1})|,
    \end{align*}
    where the inequalities are justified by (\ref{claim:counting-formula-2}), (\ref{claim:inequality}), the fact that $|\T_t^\circ(F)| < \delta|F|$, and (\ref{claim:counting-formula-1}), respectively.
    %Here we have used the inequality (\ref{claim:inequality}) to pass from $Z_n$ to $Z_{n+1}$ inside the formula. We have also used the fact noted before that $|\T_t^\circ(F)| < \delta|F|$ for any tiling $t \in \Sigma_0$, to bound the power on 2 when pulling it out of the product. 
    Taking logs and dividing through by $|F|$, we obtain %the following. 
    \begin{align*}
        h(Z_n) &< h(F,Z_n) + \delta\\
        &< \big( \delta \log |\A| + \delta \log 2 + h(F,Z_{n+1})\big) + \delta\\
        &< \delta \log |\A| + \delta \log 2 + \big(h(Z_{n+1})+\delta\big) +\delta \\
        &= h(Z_{n+1}) + 2\delta + \delta \log 2 + \delta \log |\A| \\
        &< h(Z_{n+1}) + \varepsilon,
    \end{align*}
    where we have used the property (F2) of $F$, the previous display, the property (F2) again, and our choice of $\delta$.
    This inequality establishes (\ref{claim:small-entropy-gap}). 
    %This justifies (\ref{claim:small-entropy-gap}). 
    For (\ref{claim:small-terminal-entropy}), recall that the terminal shift $Z_N$ has the property that any aligned border $\partial b \in \P^a(\partial S,Z_N)$ on any shape $S \in \S$ has exactly $1$ allowed aligned interior. Hence, we see that
    \begin{align*}
        |\P(F, Z_N)| &\leq |\A|^{\delta|F|} \cdot \sum_t\sum_{f} \prod_\tau |\operatorname{ints}^a(f(\partial\tau), Z_N)|\\
        &= |\A|^{\delta|F|} \cdot\sum_t\sum_{f} \prod_\tau 1\\
        &< |\A|^{\delta|F|} \cdot e^{\delta|F|} \cdot |\A|^{\delta|F|} \cdot 1,
    \end{align*}
    where the first inequality is justified by (\ref{claim:counting-formula-2}) and the last inequality is justified by our bounds on the number of terms in the sums (established previously).
%    Here we have used the bounds on the sizes of the ranges of the above sum operators argued before. 
Taking logs and dividing through by $|F|$, we finally have \begin{align*}
        h(Z_N) &< h(F,Z_N) + \delta\\
        &<  \big(\delta + 2\delta \log |\A|\big) + \delta\\
        &= 2\delta + 2 \delta \log |\A|\\
        &< \varepsilon,
    \end{align*}
    where we have used the property (F2) of the set $F$, the previous display, and our choice of $\delta$.
    We have now established (\ref{claim:small-terminal-entropy}). 
    
    With (\ref{claim:small-entropy-gap}) and (\ref{claim:small-terminal-entropy}) in hand, the rest of the proof is easy. %theorem follows quickly. 
    By (\ref{claim:small-entropy-gap}) and (\ref{claim:small-terminal-entropy}), we have that $(Z_n)_{n\leq N}$ is a family of subshifts of $X \times \Sigma_0$ such that $(h(Z_n))_{n\leq N}$ is $\varepsilon$-dense in $[0, h(X\times \Sigma_0)]$.
    
    For each $n \leq N$, let $X_n = \pi(Z_n) \subset X$, where $\pi$ is the projection map $\pi(x,t) = x$. From Lemma~\ref{lem:prod-sys-cond-entropy}, $\mathcal H(\pi) = h(\Sigma_0) = 0$, hence $h(X_n) = h(Z_n)$ for every $n \leq N$.
    
    Then $(X_n)_{n\leq N}$ is a descending family of subshifts of $X$ such that $(h(X_n))_{n\leq N}$ is $\varepsilon$-dense in $[0, h(X)]$. Though each $X_n$ may not be an SFT, we do know that $X$ is an SFT. One may therefore apply Theorem~\ref{sft-entropy-approx} to construct a family of SFTs $(Y_n)_{n\leq N}$ such that for each $n \leq N$, we have $X_n \subset Y_n \subset X$ and $h(X_n) \leq h(Y_n) < h(X_n) + \varepsilon$. Hence $(h(Y_n))_{n\leq N}$ is $2\varepsilon$-dense in $[0, h(X)]$.
    As $\varepsilon$ was arbitrary, we conclude that the entropies of the SFT subsystems of $X$ are dense in $[0,h(X)]$.
\end{proof}

The following ``relative" version of Theorem~\ref{main-result-sft} is stronger and easily obtained as a consequence of Theorem~\ref{main-result-sft}. % and proceeds quickly. 

\begin{corollary}
\label{main-result-sft-relative}
    Let $G$ be a countable amenable group, let $X$ be a $G$-SFT, and let $Y \subset X$ be any subsystem such that $h(Y) < h(X)$. Then \[
        \{h(Z) : Y \subset Z \subset X \text{ and $Z$ is an SFT}\}
    \] is dense in $[h(Y), h(X)]$.
\end{corollary}

\begin{proof}
    We prove the density directly. Suppose $(a,b) \subset [h(Y), h(X)]$ for positive reals $a < b$, and let $\varepsilon < (b - a)/2$. By Theorem~\ref{main-result-sft}, there exists an SFT $Z_0 \subset X$ such that $a < h(Z_0) < a + \varepsilon$. Note that these inequalities give $h(Y) < h(Z_0)$. Consider the subshift $Y\cup Z_0 \subset X$, which has entropy \[
        h(Y \cup Z_0) = \max\big(h(Y), h(Z_0)\big) = h(Z_0) \in (a,a+\varepsilon).
    \] Because $X$ is an SFT and $Y \cup Z_0 \subset X$, by Theorem~\ref{sft-entropy-approx} there is an SFT $Z$ such that $Y \cup Z_0 \subset Z \subset X$ and $h(Y \cup Z_0) \leq h(Z) < h(Y\cup Z_0) + \varepsilon$. Thus we have \[
        a < h(Y\cup Z_0) \leq h(Z) < h(Y\cup Z_0) + \varepsilon < a + 2\varepsilon < b.
    \] Since $(a,b)$ was arbitrary, the proof is complete.
\end{proof}

\section{Sofic shifts}

\subsection{An extension result for sofic shifts}

In order to address the case of sofic shifts, we seek to leverage our results on SFTs. In particular, given a sofic shift $W$, we would like an SFT $X$ such that $W$ is a factor of $X$ and such that the maximal entropy drop across the factor map is very small. The following theorem guarantees the existence of such SFTs.

\begin{theorem}
\label{sft-close-to-sofic}
    Let $W \subset \A_W^G$ be a sofic shift. For every $\varepsilon > 0$, there exists an SFT $\tilde X$ and a one-block code $\tilde\phi : \tilde X \to W$ such that {the maximal entropy gap of $\tilde\phi$ satisfies} $\mathcal H(\tilde\phi) < \varepsilon$.
\end{theorem}

\begin{proof}
    Since $W$ is sofic, there exists an SFT  $X\subset \A_X^G$ and a factor map $\phi : X \to W$. Without loss of generality, we assume that \begin{enumerate}[i.]
        \item $\phi$ is a one-block code, witnessed by the function $\Phi : \A_X \to \A_W$, and
        \item $\A_X$ and $\A_W$ are disjoint. 
    \end{enumerate}
    We abbreviate $\A_{XW} = \A_X \sqcup \A_W$. Let $\varepsilon > 0$, and select $\delta > 0$ such that \[
        4\delta + \delta(1+\delta) \log |\A_X| < \varepsilon/2.
    \] Let $K \subset G$ be a large finite subset that specifies $X$ as an SFT. The set $K$ is fixed for the remainder of this proof, and thus we shall denote $\partial_{KK^{-1}}F$ by $ \partial F$ and $\int_{KK^{-1}} F$ by $\int F$ for any finite set $F \subset G$. By Theorem~\ref{downarowicz-thm}, there exists a finite set of finite shapes $\S$ such that the following conditions are met. \begin{enumerate}[i.]
        \item Each shape $S \in \S$ is $(KK^{-1},\eta)$-invariant, where $\eta > 0$ is a constant such that $\eta|KK^{-1}| < \delta$. By Lemma~\ref{small-bdry}, this implies $|\partial S| < \delta|S|$ for each $S \in \S$.
        \item $KK^{-1} \subset S$ and $|S| > \delta^{-1}$ for each $S \in \S$.
        \item There is a point $t_0 \in \Sigma_E(\S)$ such that $h(\overline{\mathcal O}(t_0)) = 0$.
    \end{enumerate}
    
    Recall by Proposition~\ref{prop:tiling-SFT} that $\Sigma_E$ is an SFT. By Theorem~\ref{sft-entropy-approx}, there is an SFT $T$ such that $\overline{\mathcal O}(t_0) \subset T \subset \Sigma_E$ and $h(T) < h(\overline{\mathcal O}(t_0)) + \delta$. Consequently, each point $t\in T$ is an encoding of an exact tiling $\T_t$ of $G$ over $\S$ (possibly distinct from the original tiling $\T_0$), with tiling system entropy \[
        h(\T_t) = h(\overline{\mathcal O}(t)) \leq h(T) < h(\overline{\mathcal O}(t_0)) + \delta = 0 + \delta.
    \] Because $X$ and $T$ are SFTs, we have that $X \times T \subset (\A_X \times \Sigma)^G$ is also an SFT.
    
    Let $t \in T$ be arbitrary, and recall that $\T_t$ is a partition of $G$. Thus, for each $g\in G$, there is a unique tile $\tau \in \T_t$ such that $g \in \tau$. We define the notation $\T_t(g)$ by setting $\T_t(g) = \tau$.
    Next we define a map $\phi_t : X \to \A_{XW}^G$ by the following rule:  for each $g\in G$ and $x \in X$, \[
        \phi_t(x)_g = \left\{\begin{array}{ll}
            x_g & \text{if } g \in \partial \T_t(g),\\
            \Phi(x_g) & \text{if } g \in \int\T_t(g).
        \end{array}\right.
    \] This map is well-defined, as $\tau = \partial \tau \sqcup \int \tau$. The map $\phi_t$ applies the one-block code $\phi$ to ``most" of a point $x$, by relabelling the \textit{interiors} of each tile $\tau \in \T_t$.
    
    We now define a sliding block code $\varphi : X \times T \to (\A_{XW} \times \Sigma)^G$ by applying the map(s) $\phi_t$ fiber-wise: for each point $(x,t) \in X \times T$, let  \[
        \varphi(x,t) = (\phi_t(x),t).
    \]  It is {straightforward to check} that $\varphi$ is indeed a sliding block code (Definition \ref{def:block-codes}). For the theorem, the desired shift $\tilde X$ is identified with the range of this map. Let \[
        \tilde X = \varphi(X\times T) \subset (\A_{XW}\times \Sigma)^G.
    \] See Figure \ref{fig:sofic-map} for an illustration of the construction. It remains to show that there is a one-block code $\tilde\phi : \tilde X \to W$, that the shift $\tilde X$ is an SFT, and that $\mathcal H(\tilde\phi) < \varepsilon$.
    
    %%%%%%%%%%%%%%%%%%%%%%%%%%%5
    
    \begin{figure}[hbt]
    	\centering
    	\includegraphics[width=.8\textwidth]{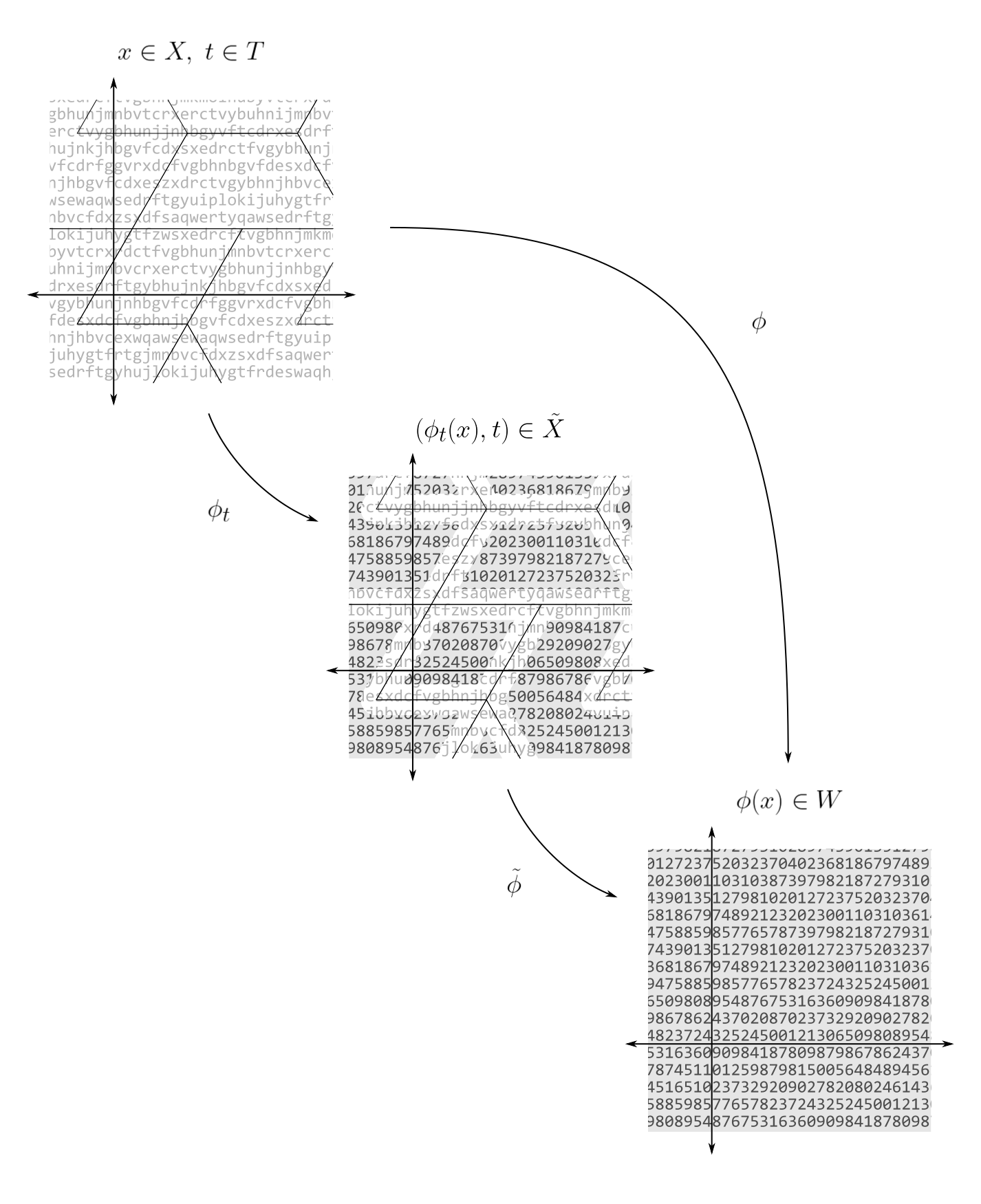}
    	\caption{A hypothetical point $x \in X$ with a tiling $t \in T$ overlayed; the partially-transformed point $\phi_t(x)$ is pictured, which is labelled with symbols from both $X$ and $W$; finally, the wholly-transformed image point $\phi(x) \in W$ is reached.}
    \label{fig:sofic-map}
    \end{figure}
    
    %%%%%%%%%%%%%%%%%%%%%%%%%%%%%%%%%%%%%%%%%%%%%%%%%%%
    
    First, let us show that $\tilde X$ factors onto $W$. The factor map is induced by the function $\tilde\Phi : \A_{XW} \to \A_W$, which is an extension of $\Phi$, defined by the following rule: $\tilde\Phi(\alpha) = \alpha$ if $\alpha\in \A_W$, and $\tilde\Phi(\alpha) = \Phi(\alpha)$ if $\alpha \in \A_X$. Let $\tilde\phi : \tilde X \to \A_W^G$ be given by\[
        \tilde\phi(\tilde x, t)_g = \tilde\Phi(\tilde x_g), \quad\forall g\in G \text{ and } \forall (\tilde x, t) \in \tilde X.
    \] 
    
    It is clear that $\tilde\phi$ is a one-block code. Let us now show that $\tilde \phi( \tilde X) = W$. Let $x \in X$ and $t \in T$, in which case $(\phi_t(x),t) \in \tilde X$ is an arbitrary point. The effect of applying the map $\phi_t$ to $x$ is to apply the one-block code $\phi$ to ``most" of $x$. The map $\tilde\phi$ then ``completes" the relabelling, via the extended function $\tilde\Phi$. In fact, we have that $\tilde\phi(\phi_t(x),t) = \phi(x) \in W$, hence $\tilde \phi (\tilde X) \subset W$. For the reverse inclusion, let $w \in W$. Since $\phi : X \to W$ is onto, there exists a point $x \in X$ such that $\phi(x) = w$. Choose $t \in T$ arbitrarily; then $(\phi_t(x),t) \in \tilde X$ and $\tilde \phi( \phi_t(x), t) = \phi(x) = w$. We conclude that $\tilde \phi : \tilde X \to W$ is a genuine factor map (and a one-block code).
    
    %%%%%%%%%%%%%%%%%%%%%%%%%%%%%%%%

    {Let us now show that} $\tilde X$ is an SFT. {We repeat that} the shift $\tilde X$ can be written in the following instructive form: % the form 
    \[
        \tilde X = \big\{
            (\phi_t(x), t) : x \in X \text{ and } t \in T
        \big\} \subset (\A_{XW} \times \Sigma)^G.
    \] {In order to show that $\tilde X$ is an SFT, we will} construct an SFT $\tilde X_1 \subset (\A_{XW} \times \Sigma)^G$ and then prove that $\tilde X = \tilde X_1$. Recall {that} $K \subset G$ specifies $X$ as an SFT. Let $K_T \subset G$ be a finite subset such that $\P(K_T, T)$ specifies $T$. {We define $\tilde X_1$ to be the set of points $(\tilde x,t) \in (\A_{XW}\times\Sigma)^G$ that satisfy the following local rules.}
    \begin{enumerate}[(R1)]
        \item Any pattern of shape $K_T$ that occurs in $t$ must belong to $\P(K_T,T)$, and any pattern of shape $K$ that occurs in $\tilde x$ and belongs to $\A_X^K$
        %- if it belongs to $\A_X^K$ - 
        must also belong to $\P(K,X)$ (recall $\P(K,\tilde X) \subset \A_{XW}^K = (\A_X \sqcup \A_W)^K$ in general). Note by Definition \ref{def:occurrence} that this condition is shift-invariant.
        
        \item For any shape $S \in \mathcal{S}$ and any $c \in G$, if $t$ satisfies $(\sigma^c t)_s = (S,s)$ for each $s \in S$, then $\exists b \in \P(S,X)$ such that $(\sigma^c \tilde x)_s = b_s \in \A_X$ for all $s \in \partial S$ and $(\sigma^c \tilde x)_s = \Phi(b_s) \in \A_W$ for all $s \in \int S$.
    \end{enumerate}
    %The reader can quickly check that these rules
    As these are local rules, they define an SFT; call it $\tilde X_1 \subset (\A_{XW} \times \Sigma)^G$. Moreover, it is easily checked that any point $(\phi_t(x),t) \in \tilde X$ satisfies these rules everywhere (by construction of $\tilde X$), and so we have $\tilde X \subset \tilde X_1$.
    
    For the reverse inclusion, consider a point $(\tilde x,t) \in \tilde X_1$. From (R1) it follows that $t \in T$, as $T$ is an SFT specified by $K_T$. Therefore, $t$ encodes an exact tiling $\T_t$ of $G$ over $\S$ with $h(\T_t) < \delta$. Let $(\tau_n)_n$ enumerate the tiles of $\T_t$, and for each $n$ let $\tau_n = S_nc_n$ for some $S_n \in \S$ and $c_n \in G$. Recall $\{\tau_n : n\in\N\}$ is a partition of $G$.
    
    Let $n \in \N$, and consider $c = c_n$ and $S = S_n$. Observe that, because $t$ encodes the tiling $\T_t$, we have $(\sigma^c t)_s = (S,s)$ for each $s \in S$. Then by (R2), there exists a block $b = b_n \in \P(S,X)$ such that $(\sigma^c \tilde x)_s = b_s$ for all $s \in \partial S$ and $(\sigma^c \tilde x)_s =\Phi(b_s)$ for all $s \in \int S$.
    
    Define a point $x\in \A_X^G$ by setting $x(\tau_n) = b_n$
    %painting $\tau_n$ with the pattern $b_n$
    for each $n \in \mathbb N$. We claim that $x$ is an allowed point of $X$ and that $\phi_t(x) = \tilde x$. Toward this, let $g \in G$ be arbitrary, and consider the translate $Kg$ (recall that $K$ specifies $X$ as an SFT).
    
    If $Kg$ intersects the \textit{interior} of any tile $\tau_n = S_n c_n$, then $Kg \subset \tau_n$ by Lemma~\ref{int-complement}. In this case, the pattern $(\sigma^g x)(K)$ is a subpattern of $b_n$, and must therefore be allowed in $X$ as $b_n \in \P(S_n,X)$. The alternative is that $Kg$ is disjoint from the interior of every tile, in which case $Kg \subset \bigcup_n \partial\tau_n$. By (R2), we also have $\tilde x_g \in \A_X$ for every $g \in \bigcup_n\int\tau_n$. In this case we have $(\sigma^g x)(K) = (\sigma^g \tilde x)(K)$, which is again allowed in $X$ by (R1).
    
    In either case we have that $(\sigma^g x)(K)$ is allowed in $X$ for any $g \in G$, and hence $x \in X$. Then by the definition of $\phi_t$, we see that  $\phi_t(x) = \tilde x$. %, from (R2) and the definition of $\phi_t$. 
    Thus, we have found a point $(x, t) \in X\times T$ such that $\varphi(x, t) = (\phi_t(x), t) = (\tilde x, t)$, and hence $(\tilde x,t) \in \tilde X$. We conclude that $\tilde X = \tilde X_1$, and therefore $\tilde X$ is an SFT.
    
    %%%%%%%%%%%%%%%%%%%%%%%%%%%%%%%%%%%%%%%%%%%%%%%
    
    %%%%%%%%%%%%%%%%%%%%%%%%%%%%%%%%%%%%%%%%%%%
    {Finally, let us show that} $\mathcal H(\tilde\phi) < \varepsilon$. Towards this end, let $\tilde X' \subset \tilde X$ be any subsystem of $\tilde X$, and let $W' = \tilde\phi(\tilde X') \subset W$. We will show %claim in advance 
    that $h(\tilde X') - h(W') < \varepsilon/2$.
    
    Let $F \subset G$ be a finite subset such that the following conditions are met. \begin{enumerate}[(F1)]
        \item $F$ is $(UU^{-1},\vartheta)$-invariant, where $U = \bigcup \S$ and $\vartheta>0$ is a constant such that $\vartheta |U||UU^{-1}| < \delta$ (recall $\delta$ was selected at the beginning of this proof). Note this implies that $F$ may be well approximated by tiles from any exact tiling of $G$ over $\S$, in the sense of Lemma~\ref{lem:invar-cond-for-tile-approx}.
        \item $F$ is large enough to $\delta$-approximate (Definition \ref{def:entropy-approx}) the entropy of the shifts $X'$, $W'$ and $T$ (recall that $h(T) < \delta$, in which case $h(F,T) < 2\delta$).
    \end{enumerate}
    Such a set exists by Proposition~\ref{prop:entropy-and-invariance-conds}. This set is fixed for the remainder of this proof. Recall that $\tilde\phi$ is a one-block code, and therefore there is a well defined map $\tilde\Phi_F : \P(F,\tilde X') \to \P(F,W')$ which takes a pattern $p \in \P(F,\tilde X')$ and applies the one-block code to $p$ (at each element of $F$).
    
    Recall also that a pattern $p \in \P(F,\tilde X')$ is of the form $p = (\phi_t(x),t)(F)$ for some points $x\in X$ and $t \in T$. The point $t$ encodes an exact tiling $\T_t$ of $G$ over $\S$. For each tile $\tau \in \T_t$, the definition of $\phi_t$ implies that \begin{equation}
        \label{observation}
        \phi_t(x)(\int \tau) \in \A_W^*, \; \text{ and } \; \phi_t(x)(\partial\tau) \in \A_X^*.
    \end{equation}
    %\textcolor{blue}{I suggest that we label equations within section so that they're all unique.}
    
    Let $q = \tilde\Phi_F(p) \in \P(F,W')$. Recall that every element $g \in F$ belongs to a unique tile $\tau = \T_t(g) \in \T_t^\times(F)$, where $\T_t^\times(F) \subset \T_t$ is the outer approximation of $F$ by the tiling $\T_t$ (Definition \ref{def:tile-approx}). By  (\ref{observation}), we have that % the following. 
    \[
        q_g = \tilde\Phi_F(p)_g = \tilde\Phi(p^{\tilde X}_g) = \bigg\{ \begin{array}{ll} \Phi(p^{\tilde X}_g) & \text{if } g\in\partial\T_t(g) \\ [0.2em] p^{\tilde X}_g & \text{if } g \in \int \T_t(g)  \end{array}
    \] 
    In particular, we have $q_g = p^{\tilde X}_g$ whenever $g$ belongs to the set \[
        F \cap \Big( \bigcup_\tau \int \tau \Big),
    \] where the union is taken over all $\tau \in \T_t^\times(F)$.
    
    In light of these observations, we are ready to estimate $|P(F,\tilde X')|$ in terms of $|P(F,W')|$. We first use $\tilde \Phi_F$ to split over $\P(F,W')$, and then we split again over all possible $T$-layers. Indeed, we have \begin{equation}
        \label{eqn:sum}
        |\P(F,\tilde X')| = \sum_q \sum_t |\{p \in \tilde{\Phi}_F^{-1}(q) : p^T = t\}|
    \end{equation} where the sums are taken over all patterns $q\in \P(F,W')$ and $t \in \P(F,T)$. Choose and fix patterns $q$ and $t$. If $p \in \P(F,\tilde X')$ is a pattern such that $\tilde \Phi_F(p) = q$ and $p^T = t$, then the observations above imply that $p$ is uniquely determined by \[
        p^{\tilde X}\Big(F \cap \Big( \bigcup_\tau \partial \tau \Big) \Big) \in \A_X^*
    \] where the union is taken over all tiles $\tau \in \T_t^\times(F)$. Moreover, our choice of $\S$ and the property (F1) of $F$ together yield that \[
        \Big| \bigcup_\tau \partial \tau \Big| < \delta \Big| \bigcup_\tau \tau \Big| = \delta|F_t^\times| < \delta(1+\delta)|F|.
    \] Therefore, there are at most $|\A_X|^{\delta(1+\delta)|F|}$ patterns $p$ such that $\tilde \Phi_F(p) = q$ and $p^T = t$. From this and Equation (\ref{eqn:sum}), we have \[
        |\P(F,\tilde X')| \leq |\P(F,W')| \cdot |\P(F,T)| \cdot |\A_X|^{\delta(1+\delta)|F|}.
    \] By taking logs and dividing through by $|F|$, we obtain the following: \begin{align*}
        h(\tilde X') &< h(F,\tilde X') + \delta\\
            &\leq \big(h(F,W') + h(F,T) + \delta(1+\delta) \log |\A_X|\big) + \delta\\
            &< \big(h(W') + \delta\big) + \big(2\delta\big) +  \delta(1+\delta) \log |\A_X| + \delta\\
            &= h(W') + 4\delta + \delta(1+\delta) \log |\A_X|\\
            &< h(W') + \varepsilon/2,
    \end{align*}
    where we have used property (F2) of the set $F$, the above inequality, property (F2) again, and our choice of $\delta$ respectively.
    {Since $\tilde X' \subset \tilde X$ was arbitrary, we have that} \[
        \mathcal H(\tilde\phi) = \sup_{\tilde X'\subset \tilde X} \big( h(\tilde X') - h(\tilde\phi(\tilde X'))\big) \leq \varepsilon/2 < \varepsilon,
    \] which completes the proof.
\end{proof}

\subsection{Subsystem entropies for sofic shifts}

{Here we present our main result concerning subsystem entropies for sofic shifts. The proof follows easily by combining our extension result (Theorem~\ref{sft-close-to-sofic}) with our result for SFTs (Theorem~\ref{main-result-sft-relative}).}

\begin{theorem}
\label{main-result-sofic-relative}
    Let $G$ be a countable amenable group, let $W$ be a sofic $G$-shift and let $V \subset W$ be any subsystem such that $h(V) < h(W)$. Then \[
        \{h(U) : V \subset U \subset W \text { and $U$ is sofic}\}
    \] is dense in $[h(V), h(W)]$.
\end{theorem}

\begin{proof}
    We prove the density directly. Let $(a,b) \subset [h(V), h(W)]$ for some real numbers $a < b$. Let $\varepsilon < (b - a)/2 < h(W) - h(V)$. By Theorem~\ref{sft-close-to-sofic}, there exists an SFT $X$ and a factor map $\phi : X \to W$ such that $\mathcal H(\phi) < \varepsilon$.
    
    Consider the preimage $Y = \phi^{-1}(V) \subset X$, which is a subshift. Note that $\phi(Y) = V$ because $\phi$ is surjective, hence $\phi|_Y : Y \to V$ is a factor map. We then have that %It then follows, 
    \[
        h(Y) \leq h(V) + \mathcal H(\phi) < h(V) + \varepsilon < h(W) \leq h(X).
    \] Note also that $b \leq h(W) \leq h(X)$ and that $a \geq h(V) > h(Y) - \varepsilon$, which together yield that $(a+\varepsilon,b) \subset [h(Y), h(X)]$. By Theorem~\ref{main-result-sft-relative}, there exists an SFT $Z$ such that $Y \subset Z \subset X$ and $h(Z) \in (a+\varepsilon,a+2\varepsilon) \subset (a,b)$. It follows that $U = \phi(Z)$ is a sofic shift for which $V \subset U \subset W$ and $h(U) \leq h(Z) < h(U) + \varepsilon$. Then we have \begin{gather*}
        a < h(Z) - \varepsilon < h(U) \leq h(Z) < a + 2\varepsilon < b.
    \end{gather*}
        Thus $h(U) \in (a,b)$, which completes the proof. % and the density follows.
\end{proof}

If one selects $V = \varnothing$ for the above theorem, then one recovers the statement that the entropies of the sofic subsystems of $W$ are dense in $[0,h(W)]$. {Next, we present our result concerning the entropies of arbitrary subsystems of sofic shifts.}

\begin{corollary}
\label{sofic-all-entropies-realized}
    Let $W$ be a sofic shift. For every nonnegative real $r \leq h(W)$, there exists a subsystem $R \subset W$ for which $h(R) = r$.
\end{corollary}

\begin{proof}
    If $h(W) = 0$, then $r = 0$, in which case one may simply select $R = W$. If $h(W) > 0$, then let $W_0 = W$ and let $(\varepsilon_n)_n$ be a sequence of positive real numbers converging to zero. We have that $W_0$ is sofic and $r \leq h(W_0)$, and without loss of generality we assume that $h(W_0) < r + \varepsilon_0$.
    
    Inductively construct a descending sequence of sofic shifts as follows. If $W_n \subset W$ is a sofic shift such that $r \leq h(W_n) < r + \varepsilon_n$, then by Theorem~\ref{main-result-sofic-relative} there exists a sofic shift $W_{n+1} \subset W_n$ for which $r \leq h(W_{n+1}) < r + \varepsilon_{n+1}$.
    
    Then $R = \bigcap_n W_n \subset W$ is a subshift such that $h(R) = \lim_n h(W_n) = r$ by Proposition~\ref{prop:entropy-limit}.
\end{proof}

\section{A counter-example}

Theorem~\ref{main-result-sofic-relative} implies that the entropies of the \textit{sofic} subsystems of a sofic shift space $W$ are dense in $[0,h(W)]$. One may wonder if this can be somehow ``sharpened"; that is, one may wonder whether % if one can say that 
\[\{h(W') : W'\subset W \text{ and $W'$ is an {SFT}} \}\] is dense in $[0, h(W)]$. However, this statement is nowhere close to true in general, as we illustrate in this section by counterexample. This example is an adaptation of a construction of Boyle, Pavlov, and Schraudner \cite{boyle_pavlov_schraudner}.

\begin{prop}
\label{prop:counterexample}
    There exists a sofic $\mathbb Z^2$-shift with positive entropy whose only SFT subsystem is a singleton.
\end{prop}

\begin{proof}
    We first construct a certain point in $\{0,1\}^{\mathbb Z}$ as the limit of a sequence of finite words, then consider the subshift it generates. Let $\delta  = 0.1$ and let $(T_n)_n$ be the sequence of natural numbers given by \[
        T_n =  2n\cdot2^n\cdot\delta^{-1} + 1
    \] for each $n$. Let $w^1 = 010 \in \{0,1\}^3$, and for each $n$ define the word \begin{equation}
    \label{eqn:recursive}
        w^{n+1} = w^n w^n \cdots w^n {w^n}\, 0^n10^n,
    \end{equation} where the $w^n$ term is repeated exactly $T_n$ times. The limit word $w^\infty \in \{0,1\}^{\mathbb N_0}$ is an infinite one-sided sequence. Define a two-sided sequence $x^* \in \{0,1\}^{\mathbb Z}$ by $x^*_i = w^\infty_{|i|}$ for each $i \in \mathbb Z$. Let $X = \overline{\mathcal O}(x^*) \subset \{0,1\}^{\mathbb Z}$ be the subshift generated by $x^*$. We claim that $X$ exhibits the following three properties.
    
    \begin{enumerate}[(P1)]
        \item $X$ is \textit{effective}, meaning that there exists a finite algorithm which enumerates a set of words $\F \subset \{0,1\}^*$ such that $X = \mathcal R\big(\{0,1\}^{\mathbb Z},\, \F\big)$.
    
        \item There exists a point $x \in X$ such that \[
            \limsup_{n\to\infty} \frac{|\{k \in [-n,n] : x_k = 1\}|}{2n+1} > 0.1.
        \]
    
        \item For each $x \in X$, either $x = 0^\mathbb Z$ or $x$ contains the word $0^n10^n$ for every $n$.
    \end{enumerate}
    
    For (P1), let $N$ be arbitrary. Note that because $X = \overline{\mathcal O}(x^*)$, any word of length $N$ occurring in any point $x \in X$ is also a word occurring in $x^*$. By the recursive definition (\ref{eqn:recursive}) and the fact that the sequence $\{T_n\}_{n=1}^{\infty}$ is recursive, there is an algorithm which, upon input $N$, prints all the words of length $N$ that do \textit{not} appear as subwords of $x^*$. The shift $X$ is therefore effective.

    %%%%%%%%%% cut below
    \iffalse
    let $N$ be arbitrary and let $F \subset \mathbb Z$ be any interval of length $N$. We note that any $F$-shaped pattern occurring in any point $x\in X$ is a pattern occurring in $x^*$ because $X = \overline{\mathcal O}(x^*)$. Moreover, any word of length $N$ occurring in $w^\infty$ is a subword of either $0^Nw^N$, $w^Nw^N$, $w^N0^N$, or $0^N10^N$. By the recurrence (\ref{eqn:recursive}), one may write an algorithm to enumerate all the subpatterns of length $N$ of these four words. In particular, one may enumerate all the words in $\{0,1\}^N$ which \textit{do not} occur as subwords of these four words. Let this collection be denoted $\mathcal F_N \subset \{0,1\}^N$, in which case $\mathcal F = \bigcup_n \mathcal F_n$ is a recursively enumerable set for which $X = \mathcal R \big( \{0,1\}^\mathbb Z, \mathcal F \big)$. This demonstrates that $X$ is effective.
    %%%%%%%%%%%% cut above
    \fi

    For (P2), we argue that $x^*$ satisfies the condition. For each $n$, let $L_n$ be the length of the word $w^n$. Note that by the recurrence (\ref{eqn:recursive}), we have \[
        L_{n+1} = T_n L_n + 2n+1 \quad \forall n.
    \] For each $n$, let $f_n$ be the \textit{frequency} of 1s in $w^n$, given by \[
        f_n = \frac{|\{ i : w^n_i = 1  \}|}{L_n}.
    \] Observe that $f_n \leq 1$ for each $n$ and $f_1 = \frac13$. It follows from the recurrence (\ref{eqn:recursive}) that \[
        f_{n+1} = \frac{f_n T_n L_n + 1}{T_nL_n + 2n+1}
    \] for each $n$. This implies that \begin{align*}
        f_n - f_{n+1} &= \frac{f_nT_nL_n + f_n(2n+1) - f_nT_nL_n - 1}{T_nL_n + 2n+1}\\
        &\leq \frac{1(2_n+1)-1}{T_n}
    \end{align*}
    
    in which case $f_n - f_{n+1} \leq \frac{2n}{T_n} < \frac{\delta}{2^n}$ for each $n$. Hence, we have that \begin{align*}
        f_1 - f_n &= (f_1 - f_2) + (f_2 - f_3) + \cdots + (f_{n-1} - f_n)\\
        &< \frac\delta 2 + \frac\delta 4 + \cdots + \frac\delta{2^{n-1}}\\[0.3em]
        &< \delta
    \end{align*} in which case $\frac13 - \delta < f_n$ for each $n$. By the recurrence (\ref{eqn:recursive}), we therefore have that \[
        \limsup_{n\to\infty} \frac{|\{ k \in [-n,n] : x^*_k = 1  \}|}{2n+1} \geq \frac13 - \delta > 0.1.
    \] and the subsequence along $(L_n)_n$ is a witness.
    
    For (P3), let $n$ be given. First, observe that the infinite sequence $w^\infty$ is the concatenation of a sequence of blocks, where each block is either the word $w^n$ or $0^m10^m$ for some $m \geq n$. Moreover, each $w^n$ begins with $0$ and ends with $0^n$. This implies that $1$s in distinct blocks are separated by at least $n+1$ appearances of the symbol $0$. Therefore, if for any $k \leq n$ we have that $10^k1$ appears anywhere in $w^\infty$, then it must appear as a subword of a single block (rather than overlapping two distinct blocks), and that block must be $w^n$. 
    
    Next, let $x \in X$ be arbitrary. If the symbol $1$ appears in $x$ at most one time, then (P3) trivially holds. Otherwise, assume that $10^k1$ appears somewhere in $x$ for some $k \geq 1$. Without loss of generality, suppose $x_0 = x_{k+1} = 1$ and $x_i = 0$ for $i \in [1,k]$. Now consider the subword $\omega = x([-L_n, L_n])$ for any $n$ such that $k < L_n$. Because $X = \overline{\mathcal O}(x^*)$, the word $\omega$ must be a subword of $x^*$. Then, either $\omega$ is a subword of $x^*([-2L_n, 2L_n])$, or $\omega$ is a subword of $w^\infty$ or a mirror reflection of one. In the first case, the definitions of $x^*$ and $w^\infty$ imply that $\omega$ contains the word $w^n$ or its mirror. In the latter two cases, the observation of the previous paragraph implies that $\omega$ must contain $w^n$ or its mirror. In any case, $0^n10^n$ is a subword of $x$. As $n$ can be made arbitrarily large, this proves (P3).
    
    We now use the shift $X$ to construct the $\mathbb Z^2$-shift which is desired for the theorem. For each point $x \in X$, let $x^\mathbb Z \in \{0,1\}^{\mathbb Z^2}$ denote the $\mathbb Z^2$-labelling given by \[
        \big(x^\mathbb Z\big)_{(i,j)} = x_i
    \] for each $(i,j) \in \mathbb Z^2$. That is, $x^{\mathbb Z}$ is a $\mathbb Z^2$-labelling such that the symbols along each column are constant, and each row is equal to $x$ itself. We shall also denote \[
        X^\mathbb Z  = \{x^\mathbb Z : x \in X\} \subset \{0,1\}^{\mathbb Z^2}.
    \] It is a theorem of Aubrun and Sablik \cite{aubrun_sablik} that if $X$ is effective, then $X^{\mathbb Z}$ is sofic.
    
    Next, consider the alphabet $\{0,1,1'\}$, where we have artificially created two independent $1$ symbols. Let $\pi : \{0,1,1'\}^{\mathbb Z^2} \to \{0,1\}^{\mathbb Z^2}$ be the one-block code which collapses 1 and $1'$. Let $Y = \pi^{-1}(X^\mathbb Z) \subset \{0,1,1'\}^{\mathbb Z^2}$. The shift $Y$ is a copy of the shift $X^\mathbb Z$, in which the 1 symbols of every point have been replaced either by $1$ or $1'$ in every possible combination.

    We claim that the shift $Y$ is the desired subshift for the theorem. Specifically, we claim that $Y$ is sofic, that $Y$ has positive entropy, and that the only nonempty SFT subsystem of $Y$ is the singleton $\{0^{\mathbb Z^2}\}$.

    To prove that $Y$ is sofic, we construct an SFT $S'$ and a factor map $\phi' : S' \to Y$ to witness the soficity of $Y$. {Since $X^{\mathbb Z}$ is sofic, there is an SFT $S \subset \A^{\mathbb Z^2}$ and a factor map $\phi : S \to X^{\mathbb Z}$.} Without loss of generality, assume that $\phi$ is a one-block code induced by the function $\Phi : \A \to \{0,1\}$.
    
    Define a new finite alphabet $\A \times \{1,1'\}$ and a one-block code $\phi'$ induced by the function $\Phi' : \A \times \{1,1'\} \to \{0,1,1'\}$ which is given by \[
        \Phi'(a,b) = \begin{cases}
            0 &\mbox{if } \Phi(a) = 0\\
            b &\mbox{if } \Phi(a) = 1
        \end{cases}
    \] for each $(a,b) \in \A \times \{1,1'\}$. 
    
    Let $S' = S \times \{1,1'\}^{\mathbb Z^2}$, which we regard as a subshift of $(\A\times\{1,1'\})^{\mathbb Z^2}$. Note that $S'$ is an SFT, because both $S$ and $\{1,1'\}^{\mathbb Z^2}$ (the full $\mathbb Z^2$-shift on two symbols) are SFTs. A point $s' \in S'$ is of the form $s' = (s,\iota)$, where $s$ is a point of $S$ and $\iota \in \{1,1'\}^{\mathbb Z^2}$ is an arbitrary $2$-coloring of $\mathbb Z^2$. The reader may easily check that $(\pi \circ \phi')(s,\iota) = \phi(s) \in X^\mathbb Z$, from which it follows that $\phi'(S') = \pi^{-1}(X^\mathbb Z) = Y$. {Then $\phi' : S' \to Y$ is a factor map. Since $S'$ is an SFT, we conclude that $Y$ is sofic.}
    
    Next, we will show that $h(Y) > 0$. From property (P2), the point $x^* \in X$ exhibits $1$s in more than $10\%$ of the positions in each of infinitely many symmetric intervals, say of the form $[-\ell_n, \ell_n]$ for an increasing sequence of natural numbers $(\ell_n)_n$. Therefore, the point $(x^*)^{\mathbb Z}$ exhibits $1$s in more than $10\%$ of the positions in each square $F_n = [-\ell_n,\ell_n]^2$. Each $1$ in the pattern $(x^*)^\mathbb Z(F_n)$ may be replaced by $1$ or $1'$ independently to yield an allowed pattern of $Y$, which implies that \[
        |\P(F_n, Y)| \geq 2^{0.1|F_n|} \quad \forall n.
    \] As $(F_n)_n$ is a F\o lner sequence for $\mathbb Z^2$, we then have $h(Y) \geq 0.1 \log 2 > 0$.
    
    It remains to show that the only nonempty SFT subsystem of $Y$ is the singleton $\{0^{\mathbb Z^2}\}$. Suppose to the contrary that $Z \subset Y$ is an SFT subsystem of $Y$ which contains a nonzero point. Since $Z$ is an SFT, we may find a constant $k \in \N$ such that the allowed patterns of $Z$ are specified by the shape $K = [0,k)^2 \subset \mathbb Z^2$. 
    
    Let $z\in Z$ be a point different from $0^{\mathbb Z^2}$ and note $\pi(z) = x^\mathbb Z \in X^\mathbb Z$ for some $x \in X$ with $x \neq 0^{\mathbb Z}$. By property (P1), the string $0^n10^n$ appears in $x$ for every $n$. Let $n > k$ be fixed. Suppose without loss of generality that $0^n10^n$ appears centered at the origin of $x$ (with $x_0 = 1$ and $x_i = 0$ for $0 < |i| \leq n$). Thus we have $z_{(0,0)} = 1$ or $1'$. In fact, by the definition of $Y$, we have $z_{(0,j)} \in \{1,1'\}$ for every $j \in \mathbb Z$.
    
    Consider the $i = 0$ column of the point $z$. Starting with each index $\ell \in \mathbb Z$ and looking up, there is a corresponding vertically oriented word $\omega^\ell \in \{1,1'\}^n$ given by $\omega^\ell_j = z_{(0,\ell+j)}$ for each $j \in [0,n)$. By the pigeonhole principle, there must exist a word $\varsigma \in \{1,1'\}^n$ such that $\varsigma = \omega^\ell$ for infinitely many choices of $\ell$. That is, for infinitely many choices of $\ell$, we have $z_{(0,\ell+j)} = \varsigma_j$ for each $j \in [0,n)$.
    
    Let $\ell_1 < \ell_2$ be two such indices where a repetition occurs, with $\ell_2 - \ell_1 > n$. That is, we have $z_{(0,\ell_1 + j)} = z_{(0,\ell_2 +j)} = \varsigma_j$ for every $j \in [0,n)$. Now consider the rectangle $r = z\big([-n,n] \times [\ell_1,\ell_2)\big)$. Tile $\mathbb Z^2$ with infinitely many translated copies of $r$ to obtain a new point $z' \in \{0,1,1'\}^{\mathbb Z^2}$. Figure \ref{fig:counterexample} illustrates the construction.
    
    \begin{figure}[bht]
        \centering
        \begin{subfigure}[h]{.4\textwidth}
            \centering
            \includegraphics[width=\textwidth]{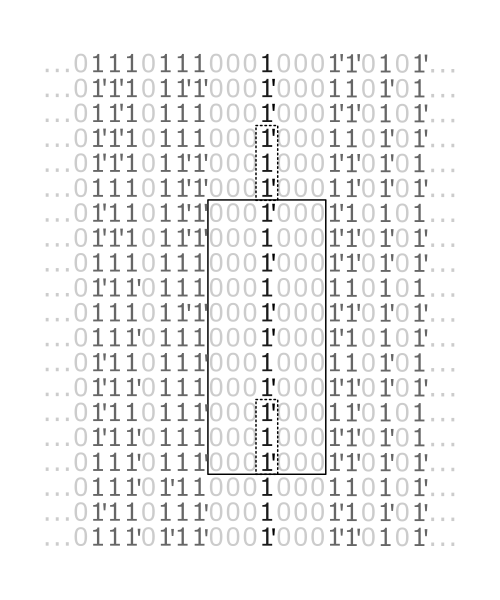}
            \caption{A hypothetical point $z$ is illustrated around $[-n,n] \times [\ell_1, \ell_2)$. The repeated vertical word $\varsigma$ is indicated by the dotted box, and the rectangle $r$ by the solid box.}
        \end{subfigure}
        \hspace{.1\textwidth}
        \begin{subfigure}[h]{.4\textwidth}
            \centering
            \includegraphics[width=\textwidth]{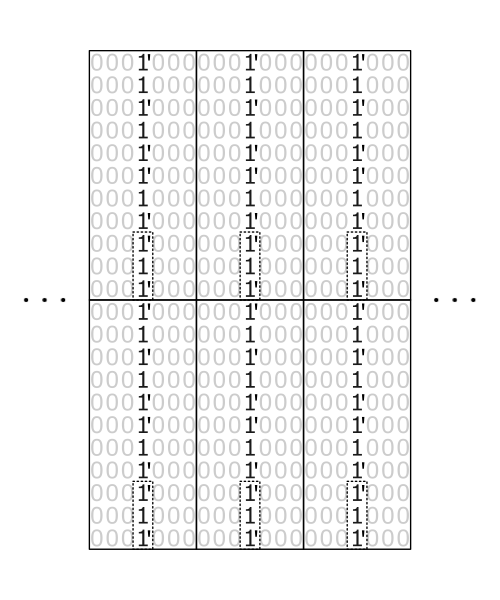}
            \caption{The rectangle $r$ is used to tile $\mathbb Z^2$ and thereby construct $z'$. Every $k\times k$ block which occurs in $z'$ also occurs in $z$.}
        \end{subfigure}
        \caption{An illustration of the construction of the contradictory point $z'$, in a hypothetical case where $n = 3$ and $k = 2$.}
    \label{fig:counterexample}
    \end{figure}
    
    Every pattern of shape $K = [0,k)^2$ which occurs in $z'$ is a pattern which occurs in $z$ (including the pattern of all zeroes), hence they are all allowed in $Z$. Because $Z$ is an SFT specified by $K$, it then follows that $z' \in Z$. Because $Z \subset Y = \pi^{-1}(X^\mathbb Z)$, there must exist a point $x' \in X$ such that $\pi(z') = (x')^{\mathbb Z}$. We obtain a contradiction, as the point $x'$ cannot satisfy the property (P3) of $X$. For instance, the word $0^{3n}10^{3n}$ cannot appear in $x'$ (as each row of $z'$ is periodic in the horizontal direction with period $2n+1 < 3n$). This demonstrates that if $Z$ is an SFT, then it contains no nonzero point. Therefore, the only nonempty SFT subsystem of $Y$ is $\{0^{\mathbb Z^2}\}$. 
    %, which has entropy $0$, hence \[
    %    \{ h(Y') : Y' \subset Y \text{ and $Y'$ is an SFT}\} = \{0\}
    %\] which is not dense in $[0,h(Y)]$ because $h(Y) > 0$.
\end{proof}

%\newpage

\printbibliography

\end{document}